\documentclass[11pt]{article}%
\usepackage{amsfonts}
\usepackage{amsmath}
\usepackage{amssymb}
\usepackage{graphicx}%
\providecommand{\U}[1]{\protect\rule{.1in}{.1in}}
\newtheorem{theorem}{Theorem}

\newtheorem{corollary}[theorem]{Corollary}

\newtheorem{lemma}[theorem]{Lemma}

\newtheorem{proposition}[theorem]{Proposition}
\newtheorem{remark}[theorem]{Remark}

\newenvironment{proof}[1][Proof]{\noindent\textbf{#1.} }{\ \rule{0.5em}{0.5em}}
{\catcode`\@=11\global\let\AddToReset=\@addtoreset
\AddToReset{equation}{section}

\AddToReset{theorem}{section}

\begin{document}

\title{Keller-Osserman estimates for some quasilinear elliptic systems }
\author{Marie-Fran\c{c}oise BIDAUT-VERON\thanks{Laboratoire de Math\'{e}matiques et
Physique Th\'{e}orique, CNRS UMR 6083, Facult\'{e} des SCiences, 37200 Tours
France. E-mail address: veronmf@univ-tours.fr}
\and Marta GARC\'{I}A-HUIDOBRO\thanks{Departamento de Matem\'{a}ticas, Pontificia
Universidad Cat\'{o}lica de Chile, Casilla 306, Correo 22, Santiago de Chile.
E-mail address: mgarcia@mat.puc.cl}
\and Cecilia YARUR\thanks{Departamento de Matem\'{a}tica y C.C., Universidad de
Santiago de Chile, Casilla 307, Correo 2, Santiago de Chile. E-mail address:
cecilia.yarur@usach.cl}}
\date{.}
\maketitle

\begin{abstract}
In this article we study quasilinear systems of two types, in a domain
$\Omega$ of $\mathbb{R}^{N}$ : with absorption terms, or mixed terms:%
\[
(A)\left\{
\begin{array}
[c]{c}%
\mathcal{A}_{p}u=v^{\delta},\\
\mathcal{A}_{q}v=u^{\mu\,},
\end{array}
\right.  \qquad(M)\left\{
\begin{array}
[c]{c}%
\mathcal{A}_{p}u=v^{\delta},\\
-\mathcal{A}_{q}v=u^{\mu\,},
\end{array}
\right.
\]
where $\delta,\mu>0$ and $1<p,q<N,$ and $D=\delta\mu-(p-1)(q-1)>0;$ the model
case is $\mathcal{A}_{p}=\Delta_{p},\mathcal{A}_{q}=\Delta_{q}.$ Despite of
the lack of comparison principle, we prove a priori estimates of
Keller-Osserman type:
\[
u(x)\leq Cd(x,\partial\Omega)^{-\frac{p(q-1)+q\delta}{D}},\qquad v(x)\leq
Cd(x,\partial\Omega)^{-\frac{q(p-1)+p\mu}{D}}.
\]
Concerning system $(M),$ we show that $v$ always satisfies Harnack inequality.
In the case $\Omega=B(0,1)\backslash\left\{  0\right\}  ,$ we also study the
behaviour near 0 of the solutions of more general weighted systems, giving a
priori estimates and removability results. Finally we prove the sharpness of
the results. \bigskip

\bigskip

\textbf{Mathematic Subject Classification (2010) }35B40, 35B45, 35J47, 35J92, 35M30

\end{abstract}

\pagebreak

\section{ Introduction\label{intro}}

In this article we study the nonnegative solutions of quasilinear systems in a
domain $\Omega$ of $\mathbb{R}^{N},$ either with absorption terms, or mixed
terms, that is,
\begin{equation}
(A)\left\{
\begin{array}
[c]{c}%
\mathcal{A}_{p}u=v^{\delta},\\
\mathcal{A}_{q}v=u^{\mu\,},
\end{array}
\right.  \qquad(M)\left\{
\begin{array}
[c]{c}%
\mathcal{A}_{p}u=v^{\delta},\\
-\mathcal{A}_{q}v=u^{\mu\,},
\end{array}
\right.  \label{dpm}%
\end{equation}
where
\[
\delta,\mu>0\quad\text{and \quad}1<p,q<N.
\]
The operators are given in divergence form by
\[
\mathcal{A}_{p}u:=div\left[  \mathrm{A}_{p}(x,u,\nabla u)\right]
,\qquad\mathcal{A}_{q}v:=div\left[  \mathrm{A}_{q}(x,v,\nabla v)\right]  ,
\]
where $\mathrm{A}_{p}$ and $\mathrm{A}_{q}$ are Carath\'{e}odory functions. In
our main results, we suppose that $\mathcal{A}_{p}$ \textit{is S-}%
$p$\textit{-C (strongly-}$p$\textit{-coercive)}, that means (see \cite{BP})
\[
\mathrm{A}_{p}(x,u,\eta).\eta\geq K_{1,p}\left\vert \eta\right\vert ^{p}\geq
K_{2,p}\left\vert \mathrm{A}_{p}(x,u,\eta)\right\vert ^{p^{\prime}}%
,\qquad\forall(x,u,\eta)\in\Omega\times\mathbb{R}^{+}\times\mathbb{R}^{N}.
\]
for some $K_{1,p},K_{2,p}>0,$ and similarly for $\mathcal{A}_{q}$. The model
type for $\mathcal{A}_{p}$ is the $p$-Laplace operator
\[
u\longmapsto\Delta_{p}u=div(|\nabla u|^{p-2}\nabla u).
\]
We prove \textit{a priori estimates of Keller-Osserman type} for such
operators, under a natural condition of "superlinearity":
\begin{equation}
\text{ }D=\delta\mu-(p-1)(q-1)>0, \label{D}%
\end{equation}
and we deduce Liouville type results of nonexistence of entire solutions. We
also study the behaviour near $0$ of nonnegative solutions of possibly
weighted systems of the form
\[
(A_{w})\left\{
\begin{array}
[c]{c}%
\mathcal{A}_{p}u=|x|^{a}\,v^{\delta},\\
\mathcal{A}_{q}v=|x|^{b}\,u^{\mu\,},
\end{array}
\right.  \qquad(M_{w})\left\{
\begin{array}
[c]{c}%
\mathcal{A}_{p}u=|x|^{a}\,v^{\delta},\\
-\mathcal{A}_{q}v=|x|^{b}\,u^{\mu\,},
\end{array}
\right.
\]
in $\Omega\backslash\left\{  0\right\}  ,$ where
\[
a,b\in\mathbb{R},\qquad a>-p,\quad b>-q.
\]
In particular we discuss about the \textit{Harnack inequality} for $u$ or
$v$.\medskip\medskip

Recall some classical results in the scalar case. For the model equation with
an absorption term
\begin{equation}
\Delta_{p}u=u^{Q}, \label{del}%
\end{equation}
in $\Omega,$ with $Q>p-1,$ the first estimate was obtained by Keller \cite{Ke}
and Osserman \cite{Os} for $p=2,$ and extended to the case $p\neq2$ in
\cite{VA}: any nonnegative solution $u\in C^{2}\left(  \Omega\right)  $
satisfies
\begin{equation}
u(x)\leq Cd(x,\partial\Omega)^{-p/(Q-p+1)}, \label{ro}%
\end{equation}
where $d(x,\partial\Omega)$ is the distance to the boundary, and $C=C(N,p,Q)$.
For the equation with a source term
\[
-\Delta_{p}u=u^{Q},
\]
up to now estimate (\ref{ro}), valid for any $Q>p-1$ in the radial case, has
been obtained only for $Q<Q^{\ast}$, where $Q^{\ast}=\frac{N(p-1)+p}{N-p}$ is
the Sobolev exponent, with difficult proofs, see \cite{GS}, \cite{BV} in the
case $p=2$ and \cite{SZ} in the general case $p>1$. For $p=2,$ the estimate,
with a universal constant, is not true for $Q=\frac{N+1}{N-3},$ and the
problem is open between $Q^{\ast}$ and $\frac{N+1}{N-3}.\medskip$

Up to our knowledge all the known estimates for systems are related with
systems for which some comparison properties hold, of competitive type, see
\cite{gmr1}, or of cooperative type, see \cite{DDGM}; or with quasilinear
operators in \cite{gm2}, \cite{WuYa}. Problems $(A)$ and $(M)$ have been the
object of very few works because such properties do not hold. The main ones
concern systems $(A_{w})$ and $(M_{w})$ in the linear case $p=q=2,$ see
\cite{BG} and \cite{BG2}; the proofs rely on the inequalities satisfied by the
mean values $\overline{u}$ and $\overline{v}$ on spheres of radius $r,$ they
cannot be extended to the quasilinear case. A radial study of system $(A)$ was
introduced in \cite{GMLS}, and recently in \cite{BGY3}.\medskip

The problem with two source terms
\[
(S)\left\{
\begin{array}
[c]{c}%
-\mathcal{A}_{p}u=|x|^{a}v^{\delta},\\
-\mathcal{A}_{q}v=|x|^{b}u^{\mu},
\end{array}
\right.
\]
was analyzed in \cite{BP}. The results are based on integral estimates, still
valid under weaker assumptions: from \cite{BP}, $\mathcal{A}_{p}$ is called
\textit{W-}$p$\textit{-C (weakly-}$p$\textit{-coercive)} if
\begin{equation}
\mathrm{A}_{p}(x,u,\eta).\eta\geq K_{p}\;\left\vert \mathrm{A}_{p}%
(x,u,\eta)\right\vert ^{p^{\prime}},\qquad\forall(x,u,\eta)\in\Omega
\times\mathbb{R}^{+}\times\mathbb{R}^{N} \label{W}%
\end{equation}
for some $K_{p}>0$; similarly for $\mathcal{A}_{q}.$ When $\delta,\mu<Q_{1},$
where $Q_{1}=\frac{N(p-1)}{N-p}$, punctual estimates were deduced for S-$p$-C,
S-$q$-C operators and it was shown that $u$ and $v$ satisfy the Harnack
inequality.$\medskip$

In \textbf{Section \ref{TOOLS}}, we give our main tools for obtaining a priori
estimates. First we show that the technique of integral estimates if
fundamental, and can be used also for systems $(A)$ and $(M).$ In Proposition
\ref{Y} we consider both equations with absorption or source terms
\begin{equation}
-\mathcal{A}_{p}u+f=0,\quad\text{or\quad\ }-\mathcal{A}_{p}u=f, \label{subsup}%
\end{equation}
in a domain $\Omega,$ where $f\in L_{loc}^{1}(\Omega),$ $f\geq0$, and obtain
local integral estimates of $f$ with respect to $u$ in a ball $B(x_{0},\rho)$.
When $\mathcal{A}_{p}$ is S-$p$-C, they imply minorizations by the W\"{o}lf
potential of $f$ in the ball
\begin{equation}
W_{1,p}^{f}(B(x_{0},\rho))=\int_{0}^{\rho}\left(  t^{p}\oint_{B(x_{0}%
,t)}f\right)  ^{\frac{1}{p-1}}\frac{dt}{t}, \label{wolf}%
\end{equation}
extending the first results of \cite{KilMa1}, \cite{KilMa2}. The second tool
is the well known weak Harnack inequalities for solutions of (\ref{subsup}) in
case of S-$p$-C operators, and a more general version in case of equation with
absorption, which appears to be very useful. The third one is a boostrap
argument given in \cite{BG} which remains essential.\medskip$\medskip$

In \textbf{Section \ref{trois}} we study both systems $(A)$ and $(M)$. When
$\mathcal{A}_{p}=\Delta_{p}$ and $\mathcal{A}_{q}=\Delta_{q},$ they admit
particular radial solutions
\[
u^{\ast}(x)=A^{\ast}\left\vert x\right\vert ^{-\gamma},v^{\ast}(r)=B^{\ast
}\left\vert x\right\vert ^{-\xi},
\]
where
\begin{equation}
\gamma=\frac{p(q-1)+q\delta}{D},\qquad\xi=\frac{q(p-1)+p\mu}{D}, \label{G}%
\end{equation}
whenever
\begin{align*}
\gamma &  >\frac{N-p}{p-1}\quad\text{and\quad\ }\xi>\frac{N-q}{q-1}%
\qquad\text{for system }(A),\\
\gamma &  >\frac{N-p}{p-1}\quad\text{and\quad}\xi<\frac{N-q}{q-1}%
\qquad\text{for system }(M).
\end{align*}
$\medskip$

Our main result for the system with absorption term $(A)$ extends precisely
the Osserman-Keller estimate of the scalar case (\ref{del}):

\begin{theorem}
\label{absorption} Assume that
\begin{equation}
\mathcal{A}_{p}\text{ is S-}p\text{-C},\quad\mathcal{A}_{q}\text{ is
S-}q\text{-C,} \label{pq}%
\end{equation}
and (\ref{D}) holds. Let $u\in W_{loc}^{1,p}\left(  \Omega\right)  \cap
C\left(  \Omega\right)  ,$ $v\in W_{loc}^{1,q}\left(  \Omega\right)  \cap
C\left(  \Omega\right)  $ be nonnegative solutions of
\[
\left\{
\begin{array}
[c]{c}%
-\mathcal{A}_{p}u+v^{\delta}\leq0,\\
-\mathcal{A}_{q}v+u^{\mu\,}\leq0,
\end{array}
\right.  \qquad\text{in }\Omega.
\]
Then for any $x\in\Omega$
\begin{equation}
u(x)\leq Cd(x,\partial\Omega)^{-\gamma},\qquad v(x)\leq Cd(x,\partial
\Omega)^{-\xi}, \label{maja}%
\end{equation}
with $C=C(N,p,q,\delta,\mu,K_{1,p},K_{2,p},K_{1,q},K_{2,q}).$
\end{theorem}

Our second result shows that the mixed system $(M)$ also satisfies the
Osserman-Keller estimate, \textit{without any restriction on }$\delta$\textit{
and }$\mu,$ and moreover the second function $v$ \textit{always satisfies
Harnack inequality:}

\begin{theorem}
\label{mixed}Assume (\ref{D}),(\ref{pq}). Let $u\in W_{loc}^{1,p}\left(
\Omega\right)  \cap C\left(  \Omega\right)  ,$ $v\in W_{loc}^{1,q}\left(
\Omega\right)  \cap C\left(  \Omega\right)  $ be nonnegative solutions of
\[
\left\{
\begin{array}
[c]{c}%
-\mathcal{A}_{p}u+v^{\delta}\leq0,\\
-\mathcal{A}_{q}v\geqq u^{\mu\,},
\end{array}
\right.  \qquad\text{in }\Omega.
\]
Then (\ref{maja}) still holds for any $x\in\Omega.\medskip$

\noindent Moreover, if $u,v$ are any nonnegative solution of system $(M),$
then $v$ satisfies Harnack inequality in $\Omega,$ and there exists another
$C>0$ as above, such that the punctual inequality holds
\begin{equation}
u^{\mu\,}(x)\leq Cv^{q-1}(x)d(x,\partial\Omega)^{-q}. \label{punc}%
\end{equation}

\end{theorem}

Notice that the results are new even for $p=q=2.$ As a consequence we deduce
Liouville properties:

\begin{corollary}
\label{Liou}Assume (\ref{D}),(\ref{pq}). Then there exist no entire
nonnegative solutions of systems $(A)$ or $(M).\medskip$
\end{corollary}

\textbf{Section \ref{quat}} concerns the behaviour near 0 of systems with
possible weights $(A_{w})$ and $(M_{w})$, where $\gamma,\xi$ are replaced by
\begin{equation}
\gamma_{a,b}=\frac{(p+a)(q-1)+(q+b)\delta}{D},\qquad\xi_{a,b}=\frac
{(q+b)(p-1)+(p+a)\mu}{D}, \label{Gab}%
\end{equation}
in other terms $\delta\xi_{a,b}=(p-1)\gamma_{a,b}+p+a,$ $\mu\gamma
_{a,b}=(q-1)\xi_{a,b}+q+b.$ We set $B_{r}=B(0,r)$ and $B_{r}^{\prime}%
=B_{r}\backslash\left\{  0\right\}  $ for any $r>0.$ Our results extend and
simplify the results of \cite{BG}, \cite{BG2} in a significant way:

\begin{theorem}
\label{weightabs}Assume (\ref{D}),(\ref{pq}). Let $u\in W_{loc}^{1,p}\left(
B_{1}^{\prime}\right)  \cap C\left(  B_{1}^{\prime}\right)  ,$ $v\in
W_{loc}^{1,q}\left(  B_{1}^{\prime}\right)  \cap C\left(  B_{1}^{\prime
}\right)  $ be nonnegative solutions of
\begin{equation}
\left\{
\begin{array}
[c]{c}%
-\mathcal{A}_{p}u+|x|^{a}\,v^{\delta}\leq0,\\
-\mathcal{A}_{q}v+|x|^{b}u^{\mu\,}\leq0,
\end{array}
\right.  \qquad\text{in }B_{1}^{\prime}. \label{dpo}%
\end{equation}
Then there exists $C=C(N,p,q,a,b,\delta,\mu,K_{1,p},K_{2,p},K_{1,q}%
,K_{2,q})>0$ such that
\begin{equation}
u(x)\leq C\left\vert x\right\vert ^{-\gamma_{a,b}},\qquad v(x)\leq C\left\vert
x\right\vert ^{-\xi_{a,b}}\qquad\text{in }B_{\frac{1}{2}}^{\prime}.
\label{ast}%
\end{equation}

\end{theorem}

\begin{theorem}
\label{weightmix}Assume (\ref{D}),(\ref{pq}). Let $u\in W_{loc}^{1,p}\left(
B_{1}^{\prime}\right)  \cap C\left(  B_{1}^{\prime}\right)  ,$ $v\in
W_{loc}^{1,q}\left(  B_{1}^{\prime}\right)  \cap C\left(  B_{1}^{\prime
}\right)  $ be nonnegative solutions of
\begin{equation}
\left\{
\begin{array}
[c]{c}%
-\mathcal{A}_{p}u+|x|^{a}\,v^{\delta}\leq0,\\
-\mathcal{A}_{q}v\geq|x|^{b}u^{\mu\,},
\end{array}
\right.  \qquad\text{in }B_{1}^{\prime}. \label{dpu}%
\end{equation}
in $B_{1}^{\prime}.$ Then there exists $C>0$ as in theorem \ref{weightabs}
such that
\begin{equation}
u(x)\leq C\left\vert x\right\vert ^{-\gamma_{a,b}},\qquad v(x)\leq
C\min(\left\vert x\right\vert ^{-\xi_{a,b}},\left\vert x\right\vert
^{-\frac{N-q}{q-1}}),\qquad\text{in }B_{\frac{1}{2}}^{\prime}. \label{oth}%
\end{equation}
Moreover if $(u,v)$ is any nonnegative solution of $(M_{w}),$ then $v$
satisfies Harnack inequality in $B_{\frac{1}{2}}^{\prime},$ and there exist
another $C>0$ as above, such that
\begin{equation}
|x|^{b+q}u^{\mu\,}(x)\leq Cv^{q-1}(x),\qquad\text{in }B_{\frac{1}{2}}^{\prime
}. \label{ath}%
\end{equation}

\end{theorem}

Moreover we give removability results for the two systems $(A_{w})$ and
$(M_{w}),$ see Theorems \ref{remabs}, \ref{remmix}, whenever $\mathcal{A}_{p}$
and $\mathcal{A}_{q}$ satisfy monotonicity and homogeneity properties,
extending to the quasilinear case \cite[Corollary 1.2]{BG} and \cite[Theorem
1.1]{BG2}.\medskip

In \textbf{Section \ref{cinq}} we show that our results on Harnack inequality
are optimal, even in the radial case. And we prove the sharpness of the
removability conditions.

\section{Main tools \label{TOOLS}}

For any $x\in\mathbb{R}^{N}$ and $r>0$, we set $B(x,r)=\left\{  y\in
\mathbb{R}^{N}/\ |y-x|<r\right\}  $ and $B_{r}=B(0,r).$

\noindent For any function $w\in L^{1}(\Omega),$ and for any weight function
$\varphi\in L^{\infty}(\Omega)$ such that $\varphi\geq0,$ $\varphi\neq0,$ we
denote by
\[
\oint_{\varphi}w=\frac{1}{\int_{\Omega}\varphi}\int_{\Omega}w\varphi
\]
the mean value of $w$ with respect to $\varphi$ and by
\[
\oint_{\Omega}w=\frac{1}{|\Omega|}\int_{\Omega}w=\oint_{1}w.
\]

\noindent For any function $g\in L_{loc}^{1}(\Omega),$ we say that a function
$u\in W_{loc}^{1,p}(\Omega)$ satisfies
\[
-\mathcal{A}_{p}u\geqq g\qquad\text{in }\Omega,\qquad(\text{resp. }%
\leqq,\text{ resp.}=)
\]
if $\mathrm{A}_{p}(x,u,\nabla u)\in L_{loc}^{p^{\prime}}(\Omega)$ and
\begin{equation}
-\int_{\Omega}\mathrm{A}_{p}(x,u,\nabla u).\nabla\phi\geqq\int_{\Omega}%
g\phi,\qquad(\text{resp. }\leqq,\text{ resp.}=) \label{ws}%
\end{equation}
for any nonnegative $\phi\in W^{1,\infty}(\Omega)$ with compact support in
$\Omega.\medskip$

\subsection{Integral estimates under weak conditions}

Next we prove integral inequalities on the second member $f$ of equations
(\ref{subsup}) in terms of the function $u,$ for either with source or with
absorption terms, obtained by multiplication by $u^{\alpha}$ with $\alpha<0$
for the source case, $\alpha>0$ for the absorption case. The method is now
classical, initiated by Serrin \cite{S2} and Trudinger \cite{Tr}, leading to
Harnack inequalities for S-$p$-C operators. These estimates were developped
for the $p$-Laplace operator in \cite{KilMa1}. Under weak conditions on the
operator, this technique of multiplication by $u^{\alpha}$ was used with
specific $f$ for obtaining Liouville results in \cite{MiPo}. It was developped
for general $f$ in \cite[Proposition 2.1]{BP} where the notion of W-$p$-C
operator was introduced. More recent Liouville results were given in
\cite[Theorem 2.1]{DAMi}, and in \cite{FaS} for the case of absorption terms.

\begin{proposition}
\label{Y}Let $\mathcal{A}_{p}$ be W-$p$-C. Let $f\in L_{loc}^{1}(\Omega),$
$f\geq0$ and let $u\in W_{loc}^{1,p}(\Omega)$ be any nonnegative solution of
inequality
\begin{equation}
-\mathcal{A}_{p}u\geqq f,\qquad\text{in }\Omega, \label{eqs}%
\end{equation}
or of inequality
\begin{equation}
-\mathcal{A}_{p}u+f\leqq0,\qquad\text{in }\Omega. \label{eqa}%
\end{equation}
Let $\xi\in\mathcal{D}(\Omega),$ with values in $\left[  0,1\right]  ,$ and
$\varphi=\xi^{\lambda},$ $\lambda>0,$ and $S_{\xi}=$supp$|\nabla\xi|.\medskip$

Then for any $\ell>p-1,$ there exists $\lambda(p,\ell)$ such that for
$\lambda\geq\lambda(p,\ell),$ there exists $C=C(N,p,K_{p},\ell,\lambda)>0$
such that%
\begin{equation}
\int_{\Omega}f\varphi\leqq C\left\vert S_{\xi}\right\vert \max_{\Omega}%
|\nabla\xi|^{p}\left(  \oint_{S_{\xi}}u^{\ell}\varphi\right)  ^{\frac
{p-1}{\ell}}. \label{inksi}%
\end{equation}

\end{proposition}

\begin{proof}
(i) First assume that $\ell>p-1+\alpha,$ with $\alpha\in\left(  1-p,0\right)
$ in case of equation (\ref{eqs}), $\alpha\in\left(  0,1\right)  $ (any
$\alpha>0$ if $u\in L_{loc}^{\infty}(\Omega))$ in case of equation
(\ref{eqa}). We claim that there exists $\lambda(p,\alpha,\ell)$ such that for
any $\lambda\geq\lambda(p,\alpha,\ell)$
\begin{equation}
\int_{\Omega}fu^{\alpha}\varphi\leqq C\left\vert S_{\xi}\right\vert
\max_{\Omega}|\nabla\xi|^{p}\left(  \oint_{S_{\xi}}u^{\ell}\varphi\right)
^{\frac{p-1+\alpha}{\ell}}, \label{inalp}%
\end{equation}
for some $C=C(N,p,K_{p},\alpha,\ell,\lambda).$ For proving (\ref{inalp}), one
can assume that $u^{\ell}\in L^{1}(B(x_{0},\rho)).$ Let $\varphi=\xi^{\lambda
}$, where $\lambda>0$ will be chosen after. Let $\delta>0,k\geq1,$ and
$(\eta_{n})$ be a sequence of mollifiers; we set $u_{\delta}=u+\delta,$
$u_{\delta,k}=\min(u,k)+\delta$ and approximate $u$ by $u_{\delta
,k,n}=u_{\delta,k}\ast\eta_{n}$, and we take $\phi=u_{\delta,k,n}^{\alpha
}\varphi$ as a test function. Then in any case, from (\ref{W}) and H\"{o}lder
inequality,
\begin{align*}
&  \left\vert \alpha\right\vert \int_{\Omega}u_{\delta,k,n}^{\alpha-1}%
\varphi\mathrm{A}_{p}(x,u,\nabla u).\nabla u_{\delta,k,n}+\int_{\Omega
}fu_{\delta,k,n}^{\alpha}\varphi\\
&  \leq\lambda\int_{S_{\xi}}u_{\delta,k,n}^{\alpha}\xi^{\lambda-1}%
|\mathrm{A}_{p}(x,u,\nabla u)||\nabla\xi|\\
&  \leq\lambda K_{p}^{-1/p^{\prime}}\int_{S_{\xi}}u_{\delta,k,n}^{\alpha}%
\xi^{\lambda-1}(\mathrm{A}_{p}(x,u,\nabla u).\nabla u)^{1/p^{\prime}}%
|\nabla\xi|\\
&  \leq\lambda K_{p}^{-1/p^{\prime}}\left(  \int_{S_{\xi}}u_{\delta
,k,n}^{\alpha-1}\xi^{\lambda}\mathrm{A}_{p}(x,u,\nabla u).\nabla u\right)
^{\frac{1}{p^{\prime}}}\left(  \int_{S_{\xi}}u_{\delta,k,n}^{\alpha+p-1}%
\xi^{\lambda-p}|\nabla\xi|^{p}\right)  ^{\frac{1}{p}}.
\end{align*}
Otherwise $(\nabla u_{\delta,k,n})$ tends to $\chi_{\left\{  u\leq k\right\}
}\nabla u$ in $L_{loc}^{p}(\Omega)$, and up to subsequence $a.e.$ in $\Omega,$
and $\mathrm{A}_{p}(x,u,\nabla u)\in L_{loc}^{p^{\prime}}(\Omega).$ By letting
$n\rightarrow\infty,$ we obtain
\begin{align*}
&  \left\vert \alpha\right\vert \int_{\left\{  u\leq k\right\}  }u_{\delta
,k}^{\alpha-1}\xi^{\lambda}\mathrm{A}_{p}(x,u,\nabla u).\nabla u+\int_{\Omega
}fu_{\delta,k}^{\alpha}\xi^{\lambda}\\
&  \leq\lambda K_{p}^{-1/p^{\prime}}\left(  \int_{S_{\xi}}u_{\delta,k}%
^{\alpha-1}\xi^{\lambda}\mathrm{A}_{p}(x,u,\nabla u).\nabla u\right)
^{\frac{1}{p^{\prime}}}\left(  \int_{S_{\xi}}u_{\delta,k}^{\alpha+p-1}%
\xi^{\lambda-p}|\nabla\xi|^{p}\right)  ^{\frac{1}{p}}\\
&  \leq\frac{\left\vert \alpha\right\vert }{2}\int_{S_{\xi}}u_{\delta
,k}^{\alpha-1}\xi^{\lambda}\mathrm{A}_{p}(x,u,\nabla u).\nabla u+C\int
_{S_{\xi}}u_{\delta,k}^{\alpha+p-1}\xi^{\lambda-p}|\nabla\xi|^{p},
\end{align*}
with $C=C(\alpha,K_{p},p,\lambda);$ otherwise, for $\alpha<1$ (or $u\in
L_{loc}^{\infty}(\Omega)$ and taking $k\geq\sup_{S_{\xi}}u)$
\begin{align*}
\int_{\Omega}u_{\delta,k}^{\alpha-1}\xi^{\lambda}\mathrm{A}_{p}(x,u,\nabla
u).\nabla u  &  =\int_{\left\{  u\leq k\right\}  }u_{\delta,k}^{\alpha-1}%
\xi^{\lambda}\mathrm{A}_{p}(x,u,\nabla u).\nabla u+\int_{\left\{  u>k\right\}
}u_{\delta,k}^{\alpha-1}\xi^{\lambda}\mathrm{A}_{p}(x,u,\nabla u).\nabla u\\
&  \leq\int_{\left\{  u\leq k\right\}  }u_{\delta,k}^{\alpha-1}\xi^{\lambda
}\mathrm{A}_{p}(x,u,\nabla u).\nabla u+Mk^{\alpha-1}%
\end{align*}
where $M=\int_{\Omega}\xi^{\lambda}\mathrm{A}_{p}(x,u,\nabla u).\nabla u$ (or
$M=0)$ is independent of $k$ and $\delta.$ Then, for any $\theta>1,$%
\begin{align*}
&  \frac{\left\vert \alpha\right\vert }{2}\int_{\left\{  u\leq k\right\}
}u_{\delta,k}^{\alpha-1}\xi^{\lambda}\mathrm{A}_{p}(x,u,\nabla u).\nabla
u+\int_{\Omega}fu_{\delta,k}^{\alpha}\xi^{\lambda}\leq C\int_{S_{\xi}%
}u_{\delta,k}^{\alpha+p-1}\xi^{\lambda-p}|\nabla\xi|^{p}+M\left\vert
\alpha\right\vert k^{\alpha-1}\\
&  \leq C\left(  \int_{S_{\xi}}u_{\delta,k}^{(\alpha+p-1)\theta}\xi^{\lambda
}\right)  ^{\frac{1}{\theta}}\left(  \int_{S_{\xi}}\xi^{\lambda-p\theta
^{\prime}}|\nabla\xi|^{p\theta^{\prime}}\right)  ^{\frac{1}{\theta^{\prime}}%
}+M\left\vert \alpha\right\vert k^{\alpha-1}.
\end{align*}
Choosing $\theta=\ell/(\alpha+p-1)>1,$ and $\lambda\geq\lambda(p,\alpha
,\ell)=p\theta^{\prime},$ we find
\begin{align*}
&  \frac{\left\vert \alpha\right\vert }{2}\int_{\left\{  u\leq k\right\}
}u_{\delta,k}^{\alpha-1}\xi^{\lambda}\mathrm{A}_{p}(x,u,\nabla u).\nabla
u+\int_{\Omega}fu_{\delta,k}^{\alpha}\xi^{\lambda}\\
&  \leq C\left(  \int_{S_{\xi}}u_{\delta,k}^{\ell}\varphi\right)
^{\frac{\alpha+p-1}{\ell}}\left(  \int_{S_{\xi}}|\nabla\xi|^{p\theta^{\prime}%
}\right)  ^{\frac{1}{\theta^{\prime}}}+M\left\vert \alpha\right\vert
k^{\alpha-1}\\
&  \leq C\left\vert S_{\xi}\right\vert ^{\frac{1}{\theta^{\prime}}}%
\max_{\Omega}|\nabla\xi|^{p}\left(  \int_{S_{\xi}}u_{\delta}^{\ell}%
\varphi\right)  ^{\frac{\alpha+p-1}{\ell}}+M\left\vert \alpha\right\vert
k^{\alpha-1},
\end{align*}
with a new constant $C=C(N,p,K,\alpha,\ell).$ As $k\rightarrow\infty,$ we
deduce
\begin{equation}
\frac{\left\vert \alpha\right\vert }{2}\int_{\Omega}u_{\delta}^{\alpha
-1}\varphi\mathrm{A}_{p}(x,u,\nabla u).\nabla u+\int_{\Omega}fu_{\delta
}^{\alpha}\varphi\leq C\left\vert S_{\xi}\right\vert ^{\frac{1}{\theta
^{\prime}}}\max_{\Omega}|\nabla\xi|^{p}\left(  \int_{S_{\xi}}u_{\delta}^{\ell
}\varphi\right)  ^{\frac{\alpha+p-1}{\ell}}. \label{chou}%
\end{equation}
Finally as $\delta\rightarrow0$ we get (\ref{inalp}) with a new constant $C.$
Moreover we deduce an estimate of the gradient terms:%
\begin{equation}
\frac{\left\vert \alpha\right\vert }{2}\int_{\Omega}u^{\alpha-1}%
\varphi\mathrm{A}_{p}(x,u,\nabla u).\nabla u\leq C\left\vert S_{\xi
}\right\vert ^{\frac{1}{\theta^{\prime}}}\max_{\Omega}|\nabla\xi|^{p}\left(
\int_{\Omega}u^{\ell}\varphi\right)  ^{\frac{\alpha+p-1}{\ell}}. \label{grat}%
\end{equation}

\noindent(ii) Next we only assume that $\ell>p-1,$ $u^{\ell}\in L^{1}%
(B(x_{0},\rho)).$ Let $\varphi$ as above, and fix some $\alpha=\alpha(p,\ell)$
such that $\alpha\in\left(  1-p,0\right)  $ and $(1-\alpha)(p-1)<\ell$ for
(\ref{eqs})$,$ $\alpha\in\left(  0,1\right)  $ and $\alpha+p-1<$ $\ell$ for
(\ref{eqa}). In any case $\tau=\ell/(1-\alpha)(p-1)>1,$ and $1/\theta
p^{\prime}+1/p\tau=(p-1)/\ell.$ Let $\lambda\geq\lambda(p,\alpha(p,\ell
),\ell)\geq p\tau^{\prime}.$ We take $\varphi$ as a test function and from
(\ref{chou}) we deduce successively, with new constants $C,$
\begin{align*}
\int_{\Omega}f\varphi &  \leq\lambda\int_{\Omega}\xi^{\lambda-1}%
|\mathrm{A}_{p}(x,u,\nabla u)|\left\vert \nabla\xi\right\vert \leq
C\int_{\Omega}\xi^{\lambda-1}|\mathrm{A}_{p}(x,u,\nabla u)|\left\vert
\nabla\xi\right\vert u_{\delta}^{\frac{\alpha-1}{p^{\prime}}}u_{\delta}%
^{\frac{1-\alpha}{p^{\prime}}}\\
&  \leq C\left(  \int_{S_{\xi}}u_{\delta}^{\alpha-1}|\mathrm{A}_{p}(x,u,\nabla
u)|^{p^{\prime}}\varphi\right)  ^{\frac{1}{p^{\prime}}}\left(  \int_{S_{\xi}%
}u_{\delta}^{(1-\alpha)(p-1)}\xi^{\lambda-p}|\nabla\xi|^{p}\right)  ^{\frac
{1}{p}}\\
&  \leq C\left(  \int_{S_{\xi}}u_{\delta}^{\alpha-1}\varphi\mathrm{A}%
_{p}(x,u,\nabla u).\nabla u\right)  ^{\frac{1}{p^{\prime}}}\left(
\int_{S_{\xi}}u_{\delta}^{\ell}\varphi\right)  ^{\frac{1}{p\tau}}\left(
\int_{S_{\xi}}\xi^{\lambda-p\tau^{\prime}}|\nabla\xi|^{p\tau^{\prime}}\right)
^{\frac{1}{p\tau^{\prime}}}\\
&  \leq C\left\vert S_{\xi}\right\vert ^{\frac{1}{\theta^{\prime}p^{\prime}%
}+\frac{1}{p\tau^{\prime}}}\max_{\Omega}|\nabla\xi|^{p}\left(
{\displaystyle\int_{S_{\xi}}}
u_{\delta}^{\ell}\varphi\right)  ^{\frac{1}{p^{\prime}\theta}+\frac{1}{p\tau}%
}\\
&  \leqq C\left\vert S_{\xi}\right\vert ^{1-\frac{p-1}{\ell}}\max_{\Omega
}|\nabla\xi|^{p}\left(
{\displaystyle\int_{S_{\xi}}}
u_{\delta}^{\ell}\varphi\right)  ^{\frac{p-1}{\ell}};
\end{align*}
and (\ref{inksi}) follows as $\delta\rightarrow0.$
\end{proof}

\begin{corollary}
\label{ball} Under the assumptions of Proposition \ref{Y}, consider any ball
$B(x_{0},2\rho)\subset\Omega,$ and any $\varepsilon\in\left(  0,\frac{1}%
{2}\right]  .$ Let $\varphi=\xi^{\lambda}$ with $\xi$ such that
\begin{equation}
\xi=1\text{ in }B(x_{0},\rho),\quad\quad\xi=0\text{ in }\Omega\backslash
\bar{B}(x_{0},\rho(1+\varepsilon))\quad\quad|\nabla\xi|\leq\frac{C_{0}%
}{\varepsilon\rho}. \label{phi}%
\end{equation}
\medskip Then for any $\ell>p-1,$ there exists $\lambda(p,\ell)>0$ such that
for $\lambda\geq\lambda(p,\ell),$ there exists $C=C(N,p,K,\ell,\lambda)>0$
such that
\begin{equation}
\oint_{\varphi}f\leq C(\varepsilon\rho)^{-p}\left(  \oint_{\varphi}u^{\ell
}\right)  ^{\frac{p-1}{\ell}}. \label{ino}%
\end{equation}

\end{corollary}

\begin{remark}
\label{part} If $S_{\xi}=\cup_{i=1}^{k}S_{\xi}^{i}$ where the $S_{\xi}^{i}$
are $2$ by $2$ disjoint, then (\ref{inksi}) can be replaced by
\begin{equation}
\int_{\Omega}f\varphi\leqq C%
{\displaystyle\sum_{i=1}^{k}}
\left\vert S_{\xi}^{i}\right\vert \max_{S_{\xi}^{i}}|\nabla\xi|^{p}\left(
\oint_{S_{\xi}^{i}}u^{\ell}\right)  ^{\frac{p-1}{\ell}}. \label{inde}%
\end{equation}

\end{remark}

\subsection{Punctual estimates under strong conditions}

When $\mathcal{A}_{p}$ is S-$p$-C, the estimate (\ref{grat}) of the gradient
is the beginning of the proof of the well-known weak Harnack inequalities:

\begin{theorem}
[\cite{S1}, \cite{Tr}](i) Let $\mathcal{A}_{p}$ be S-$p$-C, and $u\in
W_{loc}^{1,p}\left(  \Omega\right)  $ be nonnegative, such that
\[
-\mathcal{A}_{p}u\leqq0\quad\quad\text{in }\Omega;
\]
then for any ball $B(x_{0},3\rho)\subset\Omega,$ and any $\ell>p-1,$%
\begin{equation}
\sup_{B(x_{0},\rho)}u\leq C\left(  \oint_{B(x_{0},2\rho)}u^{\ell}\right)
^{\frac{1}{\ell}}, \label{uh}%
\end{equation}
with $C=C(N,p,\ell,K_{1,p},K_{2,p}).$ \medskip

\noindent(ii) Let $w\in W_{loc}^{1,p}\left(  \Omega\right)  $ be nonnegative,
such that%
\[
-\mathcal{A}_{p}w\geq0\quad\quad\text{in }\Omega;
\]
then for any ball $B(x_{0},3\rho)\subset\Omega,$ for any $\ell\in\left(
0,N(p-1)/(N-p)\right)  $
\begin{equation}
\left(  \oint_{B(x_{0},2\rho)}v^{\ell}\right)  ^{\frac{1}{\ell}}\leq
C\inf_{B(x_{0},\rho)}v. \label{wh}%
\end{equation}

\end{theorem}

Next we give a more precise version of weak Harnack inequality (\ref{uh}).
Such a kind of inequality was first established in the parabolic case in
\cite{DiBe}.

\begin{lemma}
\label{Har} Let $\mathcal{A}_{p}$ be S-$p$-C, and $u\in W_{loc}^{1,p}\left(
\Omega\right)  $ be nonnegative, such that
\[
-\mathcal{A}_{p}u\leqq0\quad\quad\text{in }\Omega;
\]
then for any $s>0$, there exists a constant $C=C(N,p,s,K_{1,p},K_{2,p})$, such
that for any ball $B(x_{0},2\rho)\subset\Omega$ and any $\varepsilon\in\left(
0,\frac{1}{2}\right]  ,$
\begin{equation}
\sup_{B(x_{0},\rho)}u\leq C\varepsilon^{-\frac{Np^{2}}{s^{2}}}\left(
\oint_{B(x_{0},\rho(1+\varepsilon))}u^{s}\right)  ^{\frac{1}{s}}. \label{hari}%
\end{equation}

\end{lemma}

\begin{proof}
From a slight adaptation of the usual case where $\varepsilon=\frac{1}{2},$
for any $\ell>p-1,$ there exists $C=C(N,\ell)>0$ such that for any
$\varepsilon\in\left(  0,\frac{1}{2}\right)  ,$
\begin{equation}
\sup_{B(x_{0},\rho)}u\leqq C\varepsilon^{-N}\left(  \oint_{B(x_{0}%
,\rho(1+\varepsilon))}u^{\ell}\right)  ^{\frac{1}{\ell}}. \label{num}%
\end{equation}
Thus we can assume $s\leq p-1.$ We fix for example $\ell=p,$ and define a
sequence $(\rho_{n})$ by $\rho_{0}=\rho,$ and $\rho_{n}=\rho(1+\frac
{\varepsilon}{2}+...+(\frac{\varepsilon}{2})^{n})$ for any $n\geq1,$ and we
set $M_{n}=\sup_{B(x_{0},\rho_{n})}u^{p}.$ From (\ref{num}) we obtain, with
new constants $C=C(N,p),$
\[
M_{n}\leqq C(\frac{\rho_{n+1}}{\rho_{n}}-1)^{-Np}\oint_{B(x_{0},\rho_{n+1}%
)}u^{p}\leq C(\frac{\varepsilon}{2})^{-(n+1)Np}\oint_{B(x_{0},\rho_{n+1}%
)}u^{p}.
\]
From the Young inequality, for any $\delta\in\left(  0,1\right)  $, and any
$r<1,$ we obtain
\begin{align*}
M_{n}  &  \leqq C(\frac{\varepsilon}{2})^{-(n+1)Np}M_{n+1}^{1-r}\oint
_{B(x_{0},\rho_{n+1})}u^{pr}\\
&  \leqq\delta M_{n+1}+r\delta^{1-1/r}(C(\frac{\varepsilon}{2})^{-(n+1)Np}%
)^{\frac{1}{r}}\left(  \oint_{B(x_{0},\rho_{n+1})}u^{pr}\right)  ^{\frac{1}%
{r}}.
\end{align*}
Defining $\kappa=$ $r\delta^{1-1/r}C{}^{\frac{1}{r}}$ and $b=(\frac
{\varepsilon}{2})^{-Np/r}$, we find
\[
M_{n}\leqq\delta M_{n+1}+b^{n+1}\kappa\left(  {\oint_{B(x_{0},\rho_{n+1})}%
}u^{pr}\right)  ^{\frac{1}{r}}.
\]
Taking $\delta=\frac{1}{2b}$ and iterating, we obtain
\begin{align*}
M_{0}  &  =\sup_{B(x_{0},\rho)}u^{p}\leqq\delta^{n+1}M_{n+1}+b\kappa%
{\displaystyle\sum_{i=0}^{n}}
(\delta b)^{i}\left(  {\oint_{B(x_{0},\rho_{n+1})}}u^{pr}\right)  ^{\frac
{1}{r}}\\
&  \leqq\delta^{n+1}M_{n+1}+2b\kappa\left(  \oint_{B(x_{0},\rho_{n+1})}%
u^{pr}\right)  ^{\frac{1}{r}}.
\end{align*}
Since $B(x_{0},\rho_{n+1})\subset B(x_{0},\rho(1+\varepsilon)),$ going to the
limit as $n\rightarrow\infty,$ and returning to $u,$ we deduce
\[
\sup_{B(x_{0},\rho)}u\leq(2b\kappa)^{1/p}\left(  {\oint_{B(x_{0}%
,\rho(1+\varepsilon))}}u^{pr}\right)  ^{\frac{1}{rp}},
\]
and the conclusion follows by taking $r=s/p.\medskip$
\end{proof}

It is interesting to make the link between Proposition \ref{Y}, with the
powerful estimates issued from the potential theory, involving
\textit{W\"{o}lf potentials}, proved in \cite{KilMa1}, \cite{KilMa2} and
\cite{KilZh}. Here we show that the lower estimates hold for any S-$p$-C operator.

\begin{corollary}
\label{wf} Suppose that $\mathcal{A}_{p}$ is S-$p$-C. Let $f\in L_{loc}%
^{1}(\Omega),$ $f\geq0$ and $u\in W_{loc}^{1,p}(\Omega)$ be any nonnegative
such that
\[
-\mathcal{A}_{p}u\geqq f,\qquad\text{in }\Omega;
\]
then for any ball $B(x_{0},2\rho)\subset\Omega,$
\begin{equation}
CW_{1,p}^{f}(B(x_{0},\rho))+\inf_{B(x_{0},2\rho)}u\leq\underset{x\rightarrow
x_{0}}{\text{ }\;\lim\inf}\;u(x),\text{ } \label{fra}%
\end{equation}
where $W_{1,p}^{f}$ is the W\"{o}lf potential of $f$ defined at (\ref{wolf}),
and $C=C(N,p,K_{1,p},K_{2,p})$. If $u$ satisfies (\ref{eqa}), then
\begin{equation}
CW_{1,p}^{f}(B(x_{0},\rho))+\underset{x\rightarrow x_{0}}{\;\lim\sup
\;}u(x)\leq\sup_{B(x_{0},2\rho)}u. \label{fri}%
\end{equation}

\end{corollary}

\begin{proof}
(i) The function $w=u-$ $m_{2\rho},$ where $m_{\rho}=\inf_{B(x_{0},\rho)}u$,
is nonnegative in $B(x_{0},2\rho),$ and satisfies the inequality
$-\mathcal{B}_{p}w\geq f,$ where
\[
w\longmapsto\mathcal{B}_{p}w=\operatorname{div}\mathrm{A}_{p}(x,w+m_{2\rho
},\nabla w)
\]
is also a S-$p$-C operator. Then from Proposition \ref{Y} with $\xi$ as in
(\ref{phi}), fixing $\ell\in\left(  0,\frac{N(p-1)}{N-p}\right)  $ and
$\varepsilon=\frac{1}{2},$ and applying Harnack inequality (\ref{wh}), there
exists $C=C(N,p,K_{1,p},K_{2,p})$ such that
\[
2C\left(  \rho^{1-N}\int_{B(x_{0},\rho)}f\right)  ^{\frac{1}{p-1}}\leq
\rho^{-1}\left(  \oint_{B(x_{0},2\rho)}(u-m_{2\rho})^{\ell}\right)  ^{\frac
{1}{\ell}}\leq\rho^{-1}(m_{\rho}-m_{2\rho}).
\]
Setting $\rho_{j}=2^{1-j}\rho,$ as in \cite{KilMa1},
\[
CW_{1,p}^{f}(B(x_{0},\rho))\leq%
{\displaystyle\sum\limits_{j=1}^{\infty}}
(m_{\rho_{j}}-m_{\rho_{j-1}})=\lim m_{\rho_{j}}-\inf_{B(x_{0},2\rho
)}u=\underset{x\rightarrow x_{0}}{\text{ }\;\lim\inf\;}u-\inf_{B(x_{0},2\rho
)}u.
\]
(ii) The function $y=M_{2\rho}-u$ where $M_{2\rho}=\sup_{B(x_{0},2\rho)}u$
satisfies the inequality $-\mathcal{C}_{p}w\geq f$ in $B(x_{0},2\rho),$ where
\[
w\longmapsto\mathcal{C}_{p}w:=\operatorname{div}\left[  \mathrm{A}%
_{p}(x,M_{2\rho}-w,\nabla w)\right]
\]
is still S-$p$-C. Then
\[
W_{1,p}^{f}(B(x_{0},\rho)\leqq C(\sup_{B(x_{0},2\rho)}u-\underset{x\rightarrow
x_{0}}{\;\lim\sup\;}u),
\]
and (\ref{fri}) follows.\bigskip
\end{proof}

\begin{remark}
The minorizations by W\"{o}lf potentials (\ref{fra}) and (\ref{fri}) have been
proved in \cite{KilMa1} and \cite{KilZh} for S-$p$-C operators of type
$\mathcal{A}_{p}u:=div\left[  \mathrm{A}_{p}(x,\nabla u)\right]  $ independent
of $u,$ satisfying moreover monotonicity and homogeneity properties, in
particular $\mathcal{A}_{p}(-u)=-\mathcal{A}_{p}u$. The solutions are defined
in the sense of potential theory, and may not belong to $W_{loc}^{1,p}\left(
\Omega\right)  ,$ $f$ can be a Radon measure; majorizations by W\"{o}lf
potentials are also given, with weighted operators, see \cite{KilMa2} and
\cite{KilZh}. In the same way Proposition \ref{Y} can also be extended to
weighted operators, see \cite[Remark 2.4]{BP} and \cite{FaS}, or to the case
of a Radon measure when $\mathcal{A}_{p}$ is S-$p$-C by using the notion of
local renormalized solution introduced in \cite{B03}.\medskip
\end{remark}

\subsection{A bootstrap result}

Finally we give a variant of a result of \cite[Lemma 2.2]{BG}:\medskip\ 

\begin{lemma}
\label{boot}Let $d,h\in\mathbb{R}$ with $d\in\left(  0,1\right)  $ and
$y,\Phi$ be two positive functions on some interval $\left(  0,R\right]
,\;$and $y$ is nondecreasing. Assume that there exist some $K,M>0$ and
$\varepsilon_{0}\in\left(  0,\frac{1}{2}\right]  \;$such that, for any
$\varepsilon\in\left(  0,\varepsilon_{0}\right]  $,
\[
y(\rho)\leqq K\varepsilon^{-h}\Phi(\rho)y^{d}\left[  \rho(1+\varepsilon
)\right]  \qquad\text{and }\max_{\tau\in\left[  \rho,3\frac{\rho}{2}\right]
}\Phi(\tau)\leqq M\text{ }\Phi(\rho),\qquad\forall\rho\in\left(  0,\frac{R}%
{2}\right]  .
\]
Then there exists $C=C(K,M,d,h,\varepsilon_{0})>0$ such that
\begin{equation}
y(\rho)\leqq C\Phi(\rho)^{\frac{1}{1-d}},\qquad\forall\rho\in\left(
0,\frac{R}{2e}\right]  . \label{jjj}%
\end{equation}

\end{lemma}

\begin{proof}
Let $\varepsilon_{m}=\varepsilon_{0}/2^{m}(m\in\mathbb{N)},$ and
$P_{m}=(1+\varepsilon_{1})..(1+\varepsilon_{m}).$ Then $(P_{m})$ has a finite
limit $P>0,$ and more precisely $P\leq e^{2\varepsilon_{0}}\leq e.$ For any
$\rho\in\left(  0,\frac{R}{2e}\right]  $ and any $m\geq1,$
\[
y(\rho P_{m-1})\leq K\varepsilon_{m}^{-h}\Phi(\rho P_{m-1})y^{d}(\rho P_{m}).
\]
By induction, for any $m\geq1,$
\[
y(\rho)\leq K^{1+d+..+d^{m-1}}\varepsilon_{1}^{-h}\varepsilon_{2}%
^{-hd}..\varepsilon_{m}^{-hd^{m-1}}\Phi(\rho)\Phi^{d}(\rho P_{1}%
)..\Phi^{d^{m-1}}(\rho P_{m-1})y^{d^{m}}(\rho P_{m}).
\]
Hence from the assumption on $\Phi$,
\[
y(\rho)\leq(K\varepsilon_{0}^{-h})^{1+d+..+d^{m-1}}2^{k(1+2d+..+md^{m-1}%
)}M^{d+2d^{2}+..+(m-1)d^{m-1}}\Phi(\rho)^{1+d+..+d^{m-1}}y^{d^{m}}(\rho
P_{m});
\]
and $y^{d^{m}}(\rho P_{m})\leq y^{d^{m}}(e\rho)\leq y^{d^{m}}(\frac{R}{2}),$
and $\lim y^{d^{m}}(\frac{R}{2})=1,$ because $d<1.$ Hence (\ref{jjj}) follows
with $C=(K\varepsilon_{0}^{-h})^{1/(1-d)}2^{h/(1-d)^{2}}M^{d/(1-d)^{2}%
}.\medskip$
\end{proof}

\section{ Keller-Osserman estimates \label{trois}}

\subsection{The scalar case}

First consider the solutions of inequality
\begin{equation}
-\mathcal{A}_{p}u+cu^{Q}\leq0,\qquad\text{in }\Omega, \label{ink}%
\end{equation}
with $Q>p-1$ and $c>0.$ From the integral estimates of Proposition \ref{Y} we
get easily Keller-Osserman estimates in the scalar case of the equation with
absorption, without any hypothesis of monotonicity on the operator:

\begin{proposition}
\label{scal}Let $Q>p-1$, $c>0.$ If $\mathcal{A}_{p}$ is S-$p$-C, and $u\in
W_{loc}^{1,p}\left(  \Omega\right)  \cap C\left(  \Omega\right)  $ is a
nonnegative solution of (\ref{ink}), there exists a constant $C=C(N,p,K_{1,p}%
,K_{2,p},Q)>0$ such that, for any $x\in\Omega,$
\begin{equation}
u(x)\leq Cc^{-1/(Q+1-p)}d(x,\partial\Omega)^{-p/(Q+1-p)}. \label{onk}%
\end{equation}

\end{proposition}

\begin{proof}
Let $B(x_{0},\rho_{0})\subset\Omega$, and $u\in W^{1,p}\left(  B(x_{0}%
,\rho_{0})\right)  .$ From Corollary \ref{ball} with $\rho\leq\frac{\rho_{0}%
}{2},$ $\varepsilon=\frac{1}{2},$ and $\ell=Q$ and a function $\varphi$
satisfying (\ref{phi}), we obtain for $\lambda=\lambda(p,Q)$
\begin{equation}
\oint_{\varphi}u^{Q}\leq c^{-1}C\rho^{-p}\left(  \oint_{\varphi}u^{Q}\right)
^{\frac{p-1}{Q}}, \label{mo}%
\end{equation}
where $C=C(N,p,K_{1,p},K_{2,p},Q).$ Then with another $C>0$ as above,
\[
\left(  \oint_{B(x_{0},\rho)}u^{Q}\right)  ^{\frac{1}{Q}}\leq Cc^{-\frac
{1}{Q+1-p}}\rho^{-\frac{p}{Q+1-p}}.
\]
Since $\mathcal{A}_{p}$ is S-$p$-C, from the weak Harnack inequality
(\ref{uh}), with another constant $C$ as above,
\[
u(x_{0})\leq C\left(  \oint_{B(x_{0},\rho)}u^{Q}\right)  ^{\frac{1}{Q}}\leq
c^{-\frac{1}{Q+1-p}}\rho^{-\frac{p}{Q+1-p}},
\]
and (\ref{onk}) follows by taking $\rho_{0}=d(x_{0},\partial\Omega).$\bigskip
\end{proof}

\subsection{The systems $(A)$ and $(M)$}

Here we prove theorems \ref{absorption}, \ref{mixed}, and Corollary
\ref{Liou}. We recall that $\gamma$ and $\xi$ are defined by (\ref{G}) under
the condition (\ref{D}) of superlinearity:
\[
\gamma=\frac{p(q-1)+q\delta}{D},\qquad\xi=\frac{q(p-1)+p\mu}{D},\qquad\text{
}D=\delta\mu-(p-1)(q-1)>0.
\]
\medskip

\begin{proof}
[Proof of Theorem \ref{absorption}]Consider a ball $B(x_{0},\rho_{0}%
)\subset\Omega,$ $\varepsilon\in\left(  0,\frac{1}{2}\right]  ,$ and a
function $\varphi$ satisfying (\ref{phi}) with $\lambda$ large enough.\medskip

(i) Case $\mu>p-1,\ \delta>q-1$. Here $C$ denotes different constants which
only depend on $N,p,q,\delta,\mu,$ and $K_{1,p},K_{2,p},K_{1,q},K_{2,q}.$ We
take $\varepsilon=\frac{1}{2}$ and apply Corollary \ref{ball} with $\rho
\leq\frac{\rho_{0}}{2}$ to the solution $u$ with $f=v^{\delta},$ and with
$\ell=\mu>p-1$. since $\mathcal{A}_{p}$ is W-$p$-C, from (\ref{ino}), we
obtain
\begin{equation}
\oint_{\varphi}v^{\delta}\leq C\rho^{-p}\left(  \oint_{\varphi}u^{\mu}\right)
^{\frac{p-1}{\mu}}, \label{m1}%
\end{equation}
and similarly we apply it to the solution $v$ with now $f=u^{\mu}$ and
$\ell=\delta>q-1:$ since $\mathcal{A}_{q}$ is W-$q$-C, we obtain
\begin{equation}
\oint_{\varphi}u^{\mu}\leq C\rho^{-q}\left(  \oint_{\varphi}v^{\delta}\right)
^{\frac{q-1}{\delta}}. \label{m2}%
\end{equation}
We can assume that $\oint_{\varphi}u^{\mu}>0.$ Indeed if $\oint_{\varphi
}u^{\mu}=0,$ then $u=0$ in $B(x_{0},\rho_{0}).$ Then $\nabla u=0,$ thus
$v^{\delta}=0$ and then the estimates are trivially verified. Replacing
(\ref{m2}) in (\ref{m1}) we deduce
\[
\oint_{\varphi}v^{\delta}\leq C\rho^{-p-q\frac{p-1}{\mu}}\left(
\oint_{\varphi}v^{\delta}\right)  ^{\frac{(q-1)(p-1)}{\mu\delta}},
\]
and similarly for $u,$ hence
\begin{equation}
\left(  \oint_{\varphi}v^{\delta}\right)  ^{\frac{1}{\delta}}\leq C\rho^{-\xi
},\qquad\left(  \oint_{\varphi}u^{\mu}\right)  ^{\frac{1}{\mu}}\leq
C\rho^{-\gamma}. \label{ingn}%
\end{equation}
Moreover, since $\mathcal{A}_{q}$ is S-$q$-C, then from the usual weak Harnack
inequality, since $v\in L_{loc}^{\infty}(\Omega),$ and $\varphi(x)=1$ in
$B(x_{0},\rho),$ with values in $\left[  0,1\right]  ,$
\[
\sup_{B(x_{0},\frac{\rho}{2})}v\leq C\left(  \oint_{B(x_{0},\rho)}v^{\delta
}\right)  ^{\frac{1}{\delta}}\leq\left(  \oint_{\varphi}v^{\delta}\right)
^{\frac{1}{\delta}}\leq C\rho^{-\xi}.
\]
Similarly%
\[
\sup_{B(x_{0},\frac{\rho}{2})}u\leq C\rho^{-\gamma},
\]
because $\mathcal{A}_{p}$ is S-$p$-C.\medskip

\noindent(ii) Case $\mu>p-1,$ and $\delta\leq q-1.$ Here we still apply
Corollary \ref{ball} with $\rho\leq\frac{\rho_{0}}{2}$, $\varepsilon\in\left(
0,1/4\right]  ,$ and a function $\varphi$ satisfying (\ref{phi}). Since
$\mu>p-1,$ we still obtain (\ref{m1}); and for any $k>q-1,$ and $\lambda$
large enough,
\begin{equation}
\oint_{\varphi}u^{\mu}\leq C(\varepsilon\rho)^{-q}\left(  \oint_{\varphi}%
v^{k}\right)  ^{(q-1)/k}, \label{mm}%
\end{equation}
and from Lemma \ref{Har},%
\[
\left(  \oint_{\varphi}v^{k}\right)  ^{1/k}\leq\sup_{B(x_{0},\rho
(1+\varepsilon))}v\leq C\varepsilon^{-\frac{Nq^{2}}{\delta^{2}}}\left(
\oint_{B(x_{0},\rho(1+2\varepsilon))}v^{\delta}\right)  ^{\frac{1}{\delta}}.
\]
Then with new constants $C,$ setting $m=q+\delta^{-2}Nq^{2}(q-1),$ and
$h=(p-1)\mu^{-1}m,$
\begin{equation}
\oint_{\varphi}u^{\mu}\leq C\varepsilon^{-m}\rho^{-q}\left(  \oint
_{B(x_{0},\rho(1+2\varepsilon))}v^{\delta}\right)  ^{\frac{(q-1)}{\delta}},
\label{jou}%
\end{equation}
hence from (\ref{m1}) and (\ref{jou}),%
\[
\oint_{B(x_{0},\rho)}v^{\delta}\leq C\oint_{\varphi}v^{\delta}\leq C\rho
^{-p}\left(  \oint_{\varphi}u^{\mu}\right)  ^{\frac{p-1}{\mu}}\leq
C\varepsilon^{-h}\rho^{-\frac{p\mu+q(p-1)}{\mu}}\left(  \oint_{B(x_{0}%
,\rho(1+2\varepsilon))}v^{\delta}\right)  ^{\frac{(p-1)(q-1)}{\delta\mu}},
\]
for any $\rho\leq\frac{\rho_{0}}{2}.$ Next we apply the boostrap Lemma
\ref{boot} with $R=\rho_{0},$ $y(\rho)=\oint_{B(x_{0},\rho)}v^{\delta}$,
$\Phi(r)=r^{-\frac{p\mu+q(p-1)}{\mu}}$ and $2\varepsilon$. We deduce that
\[
\left(  \oint_{B(x_{0},\rho)}v^{\delta}\right)  ^{1/\delta}\leq C\rho^{-\xi},
\]
for any $\rho<\frac{\rho_{0}}{2}e,$ and thus also%
\[
\sup_{B(x_{0},\frac{\rho}{2})}v\leq C\left(  \oint_{B(x_{0},\rho)}v^{\delta
}\right)  ^{\frac{1}{\delta}}\leq C\rho^{-\xi},\qquad\sup_{B(x_{0},\frac{\rho
}{2})}u\leq C\left(  \oint_{B(x_{0},\rho)}u^{\mu}\right)  ^{1/\mu}\leq
C\rho^{-\gamma}.
\]
In particular
\begin{equation}
u(x_{0})\leq C\rho_{0}^{-\gamma},\qquad v(x_{0})\leq C\rho_{0}^{-\xi},
\label{ghi}%
\end{equation}
for any ball $B(x_{0},\rho_{0})\subset\Omega,$ and the estimates (\ref{maja})
follow by taking $\rho_{0}=d(x_{0},\partial\Omega)$.\medskip\medskip
\end{proof}

\begin{proof}
[Proof of Theorem \ref{mixed}]We consider a ball $B(x_{0},\rho_{0})$ such that
$B(x_{0},2\rho_{0})\subset\Omega.$ From Proposition \ref{Y}, we have the same
estimates: for any $\ell>p-1,k>q-1,$ $\rho\leq\rho_{0},$%

\[
\oint_{\varphi}u^{\mu}\leq C\rho^{-q}\left(  \oint_{\varphi}v^{k}\right)
^{\frac{q-1}{k}},\qquad\oint_{\varphi}v^{\delta}\leq C\rho^{-p}\left(
\oint_{\varphi}u^{\ell}\right)  ^{\frac{p-1}{\ell}}.
\]
From Lemma \ref{Har} (even if $\mu<p-1),$ we have
\[
\sup_{B(x_{0},\frac{\rho}{2})}u^{\mu}\leq C\oint_{B(x_{0},\rho)}u^{\mu}.
\]
Taking $k<\frac{N(q-1)}{N-q},$ and using the weak Harnack inequality for $v,$
we obtain
\begin{align*}
\sup_{B(x_{0},\frac{\rho}{2})}u^{\mu}  &  \leq C\oint_{B(x_{0},\rho)}u^{\mu
}\leq C\oint_{\varphi}u^{\mu}\leq C\rho^{-q}\left(  \oint_{\varphi}%
v^{k}\right)  ^{\frac{q-1}{k}}\\
&  \leq C\rho^{-q}\left(  \oint_{B(x_{0},2\rho)}v^{k}\right)  ^{\frac{q-1}{k}%
}\leq C\rho^{-q}\inf_{B(x_{0},\rho)}v^{(q-1)};
\end{align*}
hence (\ref{punc}) holds in $B(x_{0},\frac{\rho}{2}).$ Moreover if
$v(x_{0})=0,$ then $u=0$ in $B(x_{0},\frac{\rho}{2}),$ then also $v=0$ in
$B(x_{0},\frac{\rho}{2}).$ Since $\Omega$ is connected, it implies that
$v\equiv0,$ and then $u\equiv0.$ If $v\not \equiv 0,$ then $v$ stays positive
in $\Omega,$ and we can write
\begin{equation}
-\mathcal{A}_{q}v=dv^{q-1},\qquad\text{in }\Omega, \label{qd}%
\end{equation}
with $d(x)=u^{\mu}/v^{(q-1)}\leq C\rho^{-q}$ in $B(x_{0},\frac{\rho}{2});$ in
particular
\begin{equation}
d(x_{0})=\frac{u^{\mu}(x_{0})}{v^{q-1}(x_{0})}\leq C\rho^{-q}, \label{vvi}%
\end{equation}
thus (\ref{punc}) holds and $v$ satisfies Harnack inequality in $\Omega:$
there exists a constant $C>0$ such that%
\[
\sup_{B(x_{0},\rho)}v\leq C\text{ }\inf_{B(x_{0},\rho)}v.
\]
Therefore%
\begin{align}
v^{\delta}(x_{0})  &  \leq\sup_{B(x_{0},\rho)}v^{\delta}\leq C\text{ }%
\inf_{B(x_{0},\rho)}v^{\delta}\leq C\oint_{\varphi}v^{\delta}\leq C\rho
^{-p}\left(  \oint_{\varphi}u^{\ell}\right)  ^{\frac{p-1}{\ell}}\nonumber\\
&  \leq C\rho^{-p}\sup_{B(x_{0},2\rho)}u^{p-1}\leq C\rho^{-p}\rho
^{-q\frac{p-1}{\mu}}\inf_{B(x_{0},4\rho)}v^{\frac{(q-1)(p-1)}{\mu}}\nonumber\\
&  \leq C\rho^{-(p+q\frac{p-1}{\mu})}v^{\frac{(q-1)(p-1)}{\mu}}(x_{0});
\label{fin}%
\end{align}
and (\ref{ghi}) follows again from (\ref{fin}) and (\ref{vvi}).\medskip
\end{proof}

\begin{remark}
Once we have proved (\ref{vvi}) we can obtain the estimate on $u$ in another
way: we have the relation in the ball
\[
\mathcal{A}_{p}u=v^{\delta}\geq cu^{\frac{\delta\mu}{q-1}\,}\qquad\text{in
}B(x_{0},\rho_{0}),
\]
with $c=C_{1}\rho_{0}^{\frac{q\delta}{q-1}};$ then from Osserman-Keller
estimates of Proposition \ref{scal} with $Q=$ $\frac{\delta\mu}{q-1}>p-1,$ we
deduce that
\[
u(x)\leq C_{2}c^{-1/Q}\rho_{0}^{-\frac{p}{Q+1-p}}=C_{3}\rho_{0}^{-\gamma
},\qquad\text{in }B(x_{0},\frac{\rho_{0}}{2}).
\]
\medskip
\end{remark}

The Liouville results are a direct consequence of the estimates:\medskip
\medskip

\begin{proof}
[Proof of Corollary \ref{Liou}]Let $x\in\mathbb{R}^{N}$ be arbitrary. Applying
the estimates in a ball $B(x,R)$, we deduce that $u(x)\leq CR^{-\gamma
},v(x)\leq CR^{-\xi}.$ Then we get $u(x)=v(x)=0$ by making $R$ tend to
$\infty.$\medskip
\end{proof}

\begin{remark}
In the scalar case of inequality (\ref{ink}) it was proved in \cite{FaS} that
the Liouville result is also valid for a W-$p$-C operator. In the case of
systems $(A)$ or $(M),$ the question is open. Indeed the method is based on
the multiplication of the inequality by $u^{\alpha}$ with $\alpha$ large
enough, and cannot be extended to the system.
\end{remark}

\section{Behaviour near an isolated point\label{quat}}

\subsection{The systems $(A_{w})$ and $(M_{w}).$}

Here we prove theorems \ref{weightabs} and \ref{weightmix}. We recall that
$\gamma_{a,b}$ and $\xi_{a,b}$ are defined by (\ref{Gab}) under condition
(\ref{D}) :
\[
\gamma_{a,b}=\frac{(p+a)(q-1)+(q+b)\delta}{D},\;\xi_{a,b}=\frac
{(q+b)(p-1)+(p+a)\mu}{D},\;D=\delta\mu-(p-1)(q-1)>0.
\]
\medskip

\begin{proof}
[Proof of Theorem \ref{weightabs}]It is a variant of Theorem \ref{absorption}:
we consider $\Omega=B_{1}^{\prime}$ and $x_{0}\in B_{\frac{1}{2}}^{\prime},$
and take $\rho_{0}=\frac{\left\vert x_{0}\right\vert }{4}$. Here we apply
Proposition \ref{Y} in the ball $B(x_{0},\rho)$ with $\rho\leq\frac{\rho_{0}%
}{2}$ and $\varepsilon\in\left(  0,\frac{1}{4}\right]  .$ The estimates
(\ref{m1}) and (\ref{mm}) are replaced by%

\begin{equation}
\oint_{\varphi}\left\vert x\right\vert ^{a}v^{\delta}\leq C(\varepsilon
\rho)^{-p}\left(  \oint_{\varphi}u^{\ell}\right)  ^{\frac{p-1}{\ell}}%
,\qquad\oint_{\varphi}\left\vert x\right\vert ^{b}u^{\mu}\leq C(\varepsilon
\rho)^{-q}\left(  \oint_{\varphi}v^{k}\right)  ^{\frac{q-1}{k}}, \label{bla}%
\end{equation}
for any $\ell>p-1,k>q-1;$ and $2\rho_{0}\leq\left\vert x\right\vert \leq
6\rho_{0}$ in $B(x_{0},2\rho_{0}),$ then in any of the cases $a\leq0$ or
$a>0,$ with a new constant $C,$%
\begin{equation}
\oint_{\varphi}v^{\delta}\leq C\varepsilon^{-p}\rho^{-(p+a)}\left(
\oint_{\varphi}u^{\ell}\right)  ^{\frac{p-1}{\ell}},\qquad\oint_{\varphi
}u^{\mu}\leq C\varepsilon^{-q}\rho^{-(q+b)}\left(  \oint_{\varphi}%
v^{k}\right)  ^{\frac{q-1}{k}}. \label{bbla}%
\end{equation}
Then all the proof is the same up to the change from $p,q$ into $p+a$ and
$q+b.$ We deduce the same estimates with $\gamma,\xi$ replaced by
$\gamma_{a,b},\xi_{a,b}:$
\begin{equation}
u(x_{0})\leq C\left\vert x_{0}\right\vert ^{-\gamma_{a,b}},\qquad v(x_{0})\leq
C\left\vert x_{0}\right\vert ^{-\xi_{a,b}}, \label{newest}%
\end{equation}
where $C$ depends on $N,p,q,a,b,\delta,\mu,$ and $K_{1,p},K_{2,p}%
,K_{1,q},K_{2,q}$.\bigskip
\end{proof}

\begin{proof}
[Proof of theorem \ref{weightmix}]In the same way we obtain estimate
(\ref{newest}), then we only need to prove the estimate with respect to
$\left\vert x\right\vert ^{-\frac{N-q}{q-1}}$. We can apply to the function
$v$ the results of \cite{B02}, recalled in \cite[Propositions 2.2 and 2.3]%
{BP}: $|x|^{b}u^{\mu\,}\in L^{1}\left(  B_{\frac{1}{2}}\right)  ,$ and for any
$k\in\left(  0,\frac{N(q-1)}{N-q}\right)  ,$ and $\rho>0$ small enough,
\begin{equation}
\left(  \oint_{B(0,\rho)}v^{k}\right)  ^{\frac{1}{k}}\leq C\rho^{-\frac
{N-q}{q-1}}. \label{lim}%
\end{equation}
Moreover, arguing as in the proof of (\ref{punc}), we obtain the punctual
inequality
\begin{equation}
u^{\mu}(x_{0})\leq C\left\vert x_{0}\right\vert ^{-(q+b)}v^{q-1}(x_{0}%
),\qquad\text{in }B_{\frac{1}{2}}^{\prime}, \label{pol}%
\end{equation}
which implies that
\[
d(x_{0})=\left\vert x_{0}\right\vert ^{b}\frac{u^{\mu}(x_{0})}{v^{q-1}(x_{0}%
)}\leq C\left\vert x_{0}\right\vert ^{-q}.
\]
Then $v$ satisfies the Harnack inequality in $B_{\frac{1}{2}}^{\prime},$
hence, from (\ref{lim}),
\[
v(x_{0})\leq\left(  \oint_{B(x_{0},\frac{\left\vert x_{0}\right\vert }{2}%
)}v^{k}\right)  ^{\frac{1}{k}}\leq C\left\vert x_{0}\right\vert ^{-\frac
{N-q}{q-1}},
\]
and (\ref{oth}) follows.
\end{proof}

\subsection{Removability results}

Here we suppose that%

\[
(C_{p})\left\{
\begin{array}
[c]{c}%
\mathcal{A}_{p}u:=div\left[  \mathrm{A}_{p}(x,\nabla u)\right]  ,\qquad
\qquad\mathcal{A}_{p}\text{ is S-}p\text{-C,}\\
\\
(\mathrm{A}_{p}(x,\xi)-\mathrm{A}_{p}(x,\zeta)).\left(  \xi-\zeta\right)
>0,\text{ }\qquad\text{for }\xi\neq\zeta,\\
\\
\mathrm{A}_{p}(x,\lambda\xi)=\left\vert \lambda\right\vert ^{p-2}%
\lambda\mathrm{A}_{p}(x,\xi),\text{ }\qquad\text{for }\lambda\neq0,
\end{array}
\right.
\]
and similarly for $\mathcal{A}_{q}.$ We give sufficient conditions ensuring
that at least one of the functions $u,v$ or both are bounded. We obtain the
two following results, relative to systems $(A_{w})$ and $(M_{w})$:

\begin{theorem}
\label{remabs} Assume (\ref{D}), $(C_{p}),(C_{q})$. Let $u\in W_{loc}%
^{1,p}\left(  B_{1}^{\prime}\right)  ,$ $v\in W_{loc}^{1,q}\left(
B_{1}^{\prime}\right)  $ be nonnegative solutions of
\[
\left\{
\begin{array}
[c]{c}%
-\mathcal{A}_{p}u+|x|^{a}\,v^{\delta}\leq0,\\
-\mathcal{A}_{q}v+|x|^{b}u^{\mu\,}\leq0,
\end{array}
\right.  \qquad\text{in }B_{1}^{\prime}.
\]
\medskip

\noindent(i) If $\gamma_{a,b}\leq\frac{N-p}{p-1},$ then $u$ is bounded near
$0;$ if $\xi_{a,b}\leq\frac{N-q}{q-1},$ then $v$ is bounded.\medskip

\noindent(ii) If moreover $\left(  u,v\right)  $ is a solution of $(A_{w})$
and $u$ is bounded near $0$ and $\delta>\frac{(p+a)(q-1)}{N-q}$ (or
$\delta=\frac{(p+a)(q-1)}{N-q}$ if $\mathcal{A}_{p}=\Delta_{p}$) then $v$ is
also bounded. In the same way if $v$ is bounded and $\mu>\frac{(q+b)(p-1)}%
{N-p}$ (or $\mu=\frac{(q+b)(p-1)}{N-p}$ if $\mathcal{A}_{q}=\Delta_{q}$) then
$u$ is also bounded.
\end{theorem}

\begin{theorem}
\label{remmix}Assume (\ref{D}), $(C_{p}),(C_{q})$. Let $u\in W_{loc}%
^{1,p}\left(  B_{1}^{\prime}\right)  \cap C\left(  B_{1}^{\prime}\right)  ,$
$v\in W_{loc}^{1,q}\left(  B_{1}^{\prime}\right)  \cap C\left(  B_{1}^{\prime
}\right)  $ be nonnegative solutions of
\[
\left\{
\begin{array}
[c]{c}%
-\mathcal{A}_{p}u+|x|^{a}\,v^{\delta}\leq0,\\
-\mathcal{A}_{q}v\geq|x|^{b}u^{\mu\,},
\end{array}
\right.  \qquad\text{in }B_{1}^{\prime}.
\]

\noindent If $\gamma_{a,b}\leq\frac{N-p}{p-1},$ or if $\gamma_{a,b}>\frac
{N-p}{p-1}$ and $\mu>\frac{(N+b)(p-1)}{N-p},$ then $u$ is bounded.
\end{theorem}

The proofs require some lemmas, adapted to subsolutions of equation
$\mathcal{A}_{p}u=0.$

\begin{lemma}
\label{supu}Assume $(C_{p}).$ Let $u\in W_{loc}^{1,p}\left(  B_{1}^{\prime
}\right)  \cap C(B_{1}^{\prime})$ be nonnegative, such that
\[
-\mathcal{A}_{p}u\leqq0\text{ \qquad in }B_{1}^{\prime}.
\]
Then, either there exists $C>0$ and $r\in\left(  0,\frac{1}{2}\right)  $ such
that
\begin{equation}
\sup_{\left\vert x\right\vert =\rho}u\geq C\rho^{\frac{p-N}{p-1}},\quad
\quad\text{for any }\rho\in\left(  0,r\right)  , \label{fil}%
\end{equation}
or $u$ is bounded near $0.$
\end{lemma}

\begin{proof}
From our assumptions on $\mathcal{A}_{p}$, there exists at least a solution
$E$ of the Dirichlet problem
\[
-\mathcal{A}_{p}E=\delta_{0},\qquad\text{in }B_{1},
\]
where $\delta_{0}$ is the Dirac mass at $0,$ in the renormalized sense, see
\cite[Theorem 3.1]{DMOP}. In particular it satisfies the equation in
$\mathcal{D}^{\prime}(B_{1}),$ and it is a smooth solution of equation
$\mathcal{A}_{p}E=0$ in $B_{1}^{\prime}.$ From \cite{S1}, \cite{S2}, there
exists $C_{1},C_{2}>0$ such that $C_{1}\left\vert x\right\vert ^{-\frac
{N-p}{p-1}}\leqq E(x)\leqq C_{2}\left\vert x\right\vert ^{-\frac{N-p}{p-1}}$
near $0.$ Assume that (\ref{fil}) does not hold. Then there exists $r_{n}%
<\min(1/n,r_{n-1})$ such that
\[
\sup_{\left\vert x\right\vert =r_{n}}u\leq\frac{1}{n}r_{n}^{\frac{p-N}{p-1}%
}\leq\frac{1}{nC_{1}}E(r_{n}).
\]
Next we use the comparison theorem in the annulus $\mathcal{C}_{n}%
\mathcal{=}\left\{  x\in\mathbb{R}^{N}:r_{n}\leq\left\vert x\right\vert
\leq\frac{1}{2}\right\}  $ for functions in $W_{loc}^{1,p}\left(
\mathcal{C}\right)  \cap C(\overline{\mathcal{C}_{n}}),$ and we find that
\[
u(x)\leq\frac{1}{nC_{1}}E(x)+\max_{\left\vert x\right\vert =\frac{1}{2}%
}u,\qquad\text{in }\mathcal{C}_{n}.
\]
Going to the limit as $n\rightarrow\infty,$ we deduce that $u$ is
bounded.\medskip
\end{proof}

Our next lemma complements the results of \cite[Proposition 2.2]{BP}:

\begin{lemma}
\label{fon}Assume that $\mathcal{A}_{p}$ is W-$p$-C. Let $f\in L_{loc}%
^{1}(B_{1}^{\prime}),f\geqq0.$ Let $u\in W_{loc}^{1,p}(B_{1}^{\prime})$ be
nonnegative, such that
\[
-\mathcal{A}_{p}u+f\leqq0\text{ \qquad in }B_{1}^{\prime}.
\]
If $\left\vert x\right\vert ^{\frac{N-p}{p-1}}u$ is bounded near $0,$ then
$f\in L_{loc}^{1}(B_{1}).$
\end{lemma}

\begin{proof}
Let $0<\rho<\frac{1}{2}.$ Here we apply Proposition \ref{Y} with $\varphi
=\xi^{\lambda}$ given by
\[
\xi=1\text{ for }\rho<\left\vert x\right\vert <\frac{1}{2},\quad\xi=0\text{
for }\left\vert x\right\vert \leqq\frac{\rho}{2}\text{or }\left\vert
x\right\vert \geqq\frac{3}{4},\quad|\nabla\xi|\leq\frac{C_{0}}{\rho}.
\]
From Remark \ref{part}, we find with for example $\ell=p,$
\begin{equation}
\int_{\rho\leqq\left\vert x\right\vert \leqq\frac{1}{2}}f\leqq C\rho
^{N-p}\left(  \oint_{\frac{\rho}{2}\leqq\left\vert x\right\vert \leqq\rho
}u^{\ell}\right)  ^{\frac{p-1}{\ell}}+C\left(  \oint_{\frac{1}{2}%
\leqq\left\vert x\right\vert \leqq\frac{3}{4}}u^{\ell}\right)  ^{\frac
{p-1}{\ell}}. \label{lam}%
\end{equation}
Hence from our assumption on $u,$ the integral is bounded, then $f\in
L^{1}(B_{\frac{1}{2}}).\medskip$
\end{proof}

\begin{proof}
[Proof of Theorem \ref{remabs}](i) Suppose that $\gamma_{a,b}\leq\frac
{N-p}{p-1}.$ Then $u(x_{0})\leq C\left\vert x_{0}\right\vert ^{-\frac
{N-p}{p-1}}.$ Let us show that $u$ is bounded. If $\gamma_{a,b}<\frac
{N-p}{p-1}$ it is a direct consequence of Lemma \ref{supu}. Then we can assume
$\gamma_{a,b}=\frac{N-p}{p-1}.$ If $u$ is not bounded, then (\ref{fil}) holds
for some $C>0.$ Let us set $f=\left\vert x\right\vert ^{a}v^{\delta}.$ From
(\ref{bbla}) with $\varepsilon=\frac{1}{4}$ then for any $r_{0}\leq\frac{1}%
{2}$ and any $x_{0}$ such that $\left\vert x_{0}\right\vert =r_{0},$ and Lemma
\ref{Har}, taking $\rho=\frac{r_{0}}{4},$
\begin{align*}
u^{\mu}(x_{0})  &  \leq C\oint_{B(x_{0},\rho)}u^{\mu}\leq Cr_{0}%
^{-(q+b)-N\frac{q-1}{\delta}}\left(  \int_{B(x_{0},2\rho)}v^{\delta}\right)
^{\frac{q-1}{\delta}}\\
&  \leq Cr_{0}^{-(q+b)-(N+a)\frac{q-1}{\delta}}\left(  \int_{\frac{r_{0}}%
{2}\leqq\left\vert x\right\vert \leqq\frac{3r_{0}}{2}}f\right)  ^{\frac
{q-1}{\delta}},
\end{align*}
then
\[
Cr_{0}^{-\mu\gamma_{a,b}}=Cr_{0}^{-(q-1)\xi_{a,b}-q-b}\leq\sup_{\left\vert
x\right\vert =r_{0}}u^{\mu}\leq Cr_{0}^{-(q+b)-(N+a)\frac{q-1}{\delta}}\left(
\int_{\frac{r_{0}}{2}\leqq\left\vert x\right\vert \leqq\frac{3r_{0}}{2}%
}f\right)  ^{\frac{q-1}{\delta}},
\]%
\[
Cr_{0}^{-(q-1)\xi_{a,b}\frac{\delta}{q-1}+(N+a)}=Cr_{0}^{0}=C\leq\int
_{\frac{r_{0}}{2}\leqq\left\vert x\right\vert \leqq\frac{3r_{0}}{2}}f;
\]
then for any $n\in\mathbb{N},$
\[
C\leq\int_{\frac{r_{0}}{2.3^{n}}\leqq\left\vert x\right\vert \leqq\frac{r_{0}%
}{2.3^{n-1}}}f.
\]
By summation it contradicts Lemma \ref{fon}. Similarly for $v.\medskip
\medskip$

\noindent(ii) Suppose that $(u,v)$ is a solution of $(A_{w})$ and $u$ is
bounded and $\delta\geq\frac{(p+a)(q-1)}{N-q}.$ Here $v$ satisfies equation
$\mathcal{A}_{q}v=g$ with $g=\left\vert x\right\vert ^{b}u^{\mu}\leqq
C\left\vert x\right\vert ^{b},$ thus $g\in L^{N/q+\varepsilon}(\Omega)$ for
some $\varepsilon>0,$ then from \cite{S1}, \cite{S2}, if $v$ is not bounded
near 0, then there exist $C_{1},C_{2}>0$ such that
\[
C_{1}\left\vert x\right\vert ^{-\frac{N-q}{q-1}}\leqq v\leqq C_{2}\left\vert
x\right\vert ^{-\frac{N-q}{q-1}}%
\]
near $0$. If $\delta>\frac{(p+a)(q-1)}{N-q}$ then
\[
\mathcal{A}_{p}u=|x|^{a}v^{\delta}\geq C_{1}|x|^{a-\delta\frac{N-q}{q-1}%
}=C_{1}|x|^{-p-\varepsilon},
\]
for some $\varepsilon>0,$ then from (\ref{bla}),
\[
\rho^{-p-\varepsilon}\leqq C\oint_{\varphi}|x|^{-p-\varepsilon}\leq C\rho
^{-p}\left(  \oint_{\varphi}u^{\ell}\right)  ^{\frac{p-1}{\ell}}\leqq
C\rho^{-p},
\]
which is a contradiction. If $\delta=\frac{(p+a)(q-1)}{N-q},$ then
\[
C_{2}|x|^{-p}\geq\mathcal{A}_{p}u=|x|^{a}v^{\delta}\geq C_{1}|x|^{-p}.
\]
Otherwise $u$ is bounded by some $M$ in a ball $B_{r}^{\prime}.$ Then the
function $w=M-u$ is nonnegative and bounded and satisfies
\[
-\mathcal{A}_{p}w\geq C_{1}|x|^{-p}\qquad\text{in }B_{r}^{\prime}.
\]
But for $\mathcal{A}_{p}=\Delta_{p},$ there is no bounded solution of this
inequality, from \cite[Proposition 2.7]{BP}, we reach a
contradiction.$\medskip$
\end{proof}

\begin{remark}
The results obviously apply to the scalar case, finding again and improving a
result of \cite{VaV}.$\medskip$
\end{remark}

\begin{proof}
[Proof of Theorem \ref{remmix}](i) Assume $\gamma_{a,b}\leq\frac{N-p}{p-1}.$
The proof of part (i) of Theorem \ref{remabs} is still valid and shows that
$u$ is bounded.$\medskip$

\noindent(ii) Assume $\gamma_{a,b}>\frac{N-p}{p-1}$ and $\mu>\frac
{(N+b)(p-1)}{N-p}.$ Then $\xi_{a,b}>\frac{N-q}{q-1},$ thus the estimate
(\ref{oth}) for $v$ gives $v(x_{0})\leq C\left\vert x_{0}\right\vert
^{-\frac{N-q}{q-1}}$, then
\[
u^{\mu}(x_{0})\leq C\left\vert x_{0}\right\vert ^{-(q+b)}v^{(q-1)}(x_{0})\leq
C\left\vert x_{0}\right\vert ^{-(N+b)}.
\]
Then $\rho^{\frac{N-p}{p-1}}\sup_{\left\vert x\right\vert =\rho}u$ tends to
$0,$ hence $u$ is bounded from Lemma \ref{supu}.
\end{proof}

\begin{remark}
Let us give an alternative proof of (i): the punctual inequality (\ref{pol})
implies that near 0,
\[
\mathcal{A}_{p}u\geq|x|^{a}v^{\delta}\geq C|x|^{a+\delta(q+b)/(q-1)}%
u^{\mu\delta/(q-1)};
\]
then we are reduced to a simple scalar inequality:
\begin{equation}
-\mathcal{A}_{p}u+|x|^{m}u^{Q}\leq0, \label{hij}%
\end{equation}
with $Q=\frac{\mu\delta}{q-1}>p-1$ and $m=a+\frac{\delta(q+b)}{q-1}>-p.$ And
$\gamma_{a,b}=\frac{m+p}{Q+1-p}\leq\frac{N-p}{p-1};$ applying Theorem
\ref{remabs} to the scalar inequality (\ref{hij}), we find again that $u$ is bounded.
\end{remark}

\section{Sharpness of the results\label{cinq}}

In this last section we show the optimality of our results by constructing
some radial solutions of systems $(A_{w})$ or $(M_{w})$ in case $\mathcal{A}%
_{p}=\Delta_{p},\mathcal{A}_{q}=\Delta_{q}.$ They are based on the
transformation introduced in \cite{BGi}, valid for systems with any sign:
\[
\left\{
\begin{array}
[c]{c}%
-\Delta_{p}u=-\operatorname{div}(\left\vert \nabla u\right\vert ^{p-2}\nabla
u)=\varepsilon_{1}\left\vert x\right\vert ^{a}v^{\delta},\\
-\Delta_{q}v=-\operatorname{div}(\left\vert \nabla v\right\vert ^{q-2}\nabla
u)=\varepsilon_{2}\left\vert x\right\vert ^{b}u^{\mu},
\end{array}
\right.
\]
with $\varepsilon_{1}=-1=\varepsilon_{2}$ for the system with absorption, and
$\varepsilon_{1}=-1,\varepsilon_{2}=1$ for the mixed system: setting
\[
X(t)=-\frac{ru^{\prime}}{u},\quad Y(t)=-\frac{rv^{\prime}}{v},\quad
Z(t)=-\varepsilon_{1}r^{1+a}u^{s}v^{\delta}\frac{u^{\prime}}{\left\vert
u^{\prime}\right\vert ^{p}},\quad W(t)=-\varepsilon_{2}r^{1+b}u^{\mu}%
v^{m}\frac{v^{\prime}}{\left\vert v^{\prime}\right\vert ^{q}},
\]
where $t=\ln r,$ and we obtain the system
\[
(\Sigma)\left\{
\begin{array}
[c]{c}%
X_{t}=X\left[  X-\frac{N-p}{p-1}+\frac{Z}{p-1}\right]  ,\\
Y_{t}=Y\left[  Y-\frac{N-q}{q-1}+\frac{W}{q-1}\right]  ,\\
Z_{t}=Z\left[  N+a-\delta Y-Z\right]  ,\\
W_{t}=W\left[  N+b-\mu X-W\right]  .
\end{array}
\right.
\]
And $u,v$ are recovered from $X,Y,Z,W$ by the relations
\begin{equation}
u\mathbf{=}r^{-\gamma_{a,b}}\mathbf{(}\left\vert X\right\vert ^{p-1}%
Z)^{(q-1)/D}\mathbf{(}\left\vert Y\right\vert ^{q-1}W)^{\delta/D}%
,\mathbf{\qquad}v\mathbf{=}r^{-\xi_{a,b}}\mathbf{(}\left\vert X\right\vert
^{p-1}Z)^{\mu/D}\mathbf{(}\left\vert Y\right\vert ^{q-1}W)^{(p-1)/D}.
\label{form}%
\end{equation}

\subsection{About Harnack inequality}

Here we show that Harnack inequality can be \textit{false} in case of system
$(A_{w})$ and also for the function $u$ of system $(M_{w}),$ even in the
radial case; indeed we construct nonnegative radial solutions of system
$(A_{w})$ in a ball such that $u(0)=0<v(0),$ or by symmetry $u(0)>0=v(0)$ and
solutions of system $(M_{w})$ such that $u(0)=0<v(0).$ Such solutions were
constructed in \cite{GMLS} by using Schauder theorem, and in \cite{BGY3} in
the case of system $(A_{w})$ for $p=q=2$ by using system $(\Sigma).$ Here we
show that the construction of \cite{BGY3} extends to the general case. We
consider the radial regular solutions, which are $C^{2}$ if $a,b\geq0,$ and
$C^{1}$ if $a,b>-1.$

\begin{proposition}
\label{exi} Suppose that $\mathcal{A}_{p}=\Delta_{p}$ and $\mathcal{A}%
_{q}=\Delta_{q}.$ For any $v_{0}>0,$ there exists a regular radial solution of
$(A_{w})$ and $(M_{w})$ such that $u(0)=0<v(0)=v_{0}.$
\end{proposition}

\begin{proof}
The regular solutions $(u,v)$ with nonnegative initial data $(u_{0},v_{0}%
)\neq\left(  0,0\right)  $ are increasing for system $(A_{w})$, hence
$X,Y<0<Z,W$ and $u$ is increasing and $v$ is decreasing for system $(M_{w})$,
hence $X<0<Y$ and $Z,W>0.$ As shown in \cite{BGi}, the solutions $(u,v)$ with
$u(0)=u_{0}>0$ and $v(0)=v_{0}>0$ correspond to the trajectories of system
($\Sigma$) converging to the fixed point $N_{0}=(0,0,N+a,N+b)$ as
$t\longrightarrow-\infty,$ and local existence and uniqueness holds as in
\cite[Proposition 4.4]{BGi}. As in \cite{BGY3} the solutions such that
$u_{0}=0<v_{0}$ correspond to a trajectory converging to the point
$S_{0}=\left(  \bar{X},0,\bar{Z},\bar{W}\right)  =\left(  -\frac{p+a}%
{p-1},0,N+a,N+b+\mu\frac{p+a}{p-1}\right)  .$ The linearization at $S_{0}$
gives the eigenvalues
\[
\lambda_{1}=\bar{X}<0,\quad\lambda_{2}=\frac{1}{q-1}(q+b+\mu\frac{p+a}%
{p-1})>0,\quad\lambda_{3}=-\bar{Z}<0,\quad\lambda_{4}=-\bar{W}<0.
\]
Then the unstable manifold $\mathcal{V}_{u}$ has dimension 1 and
$\mathcal{V}_{u}\cap\left\{  Y=0\right\}  =\emptyset$, thus there exists a
unique trajectory such that $Y<0$ (resp. $Y>0)$ and $Z,W>0$. There holds
$\lim_{t\rightarrow-\infty}e^{-\lambda_{2}t}Y=c>0,$ $\lim X=\bar{X},$ $\lim
Z=\bar{Z},$ $\lim W=\bar{W},$ then from (\ref{form}) $v\mathbf{\ }$ has a
positive limit $v_{0}$, and $u$ tends to 0. By scaling we obtain the existence
and uniqueness of solutions for any $v_{0}>0$.
\end{proof}

\subsection{About removability}

Here also we show that the results of Theorems \ref{remabs} and \ref{remmix}
are optimal, by constructing singular solutions when the assumptions are not
satisfied. We begin by system $(A_{w}),$ extending \cite[Proposition
3.2]{BGY3}. Obviously it admits a particular singular solution when
$\gamma_{a,b}>\frac{N-p}{p-1}$ and $\xi_{a,b}>\frac{N-q}{q-1}.$ Moreover we
find other types of singular solutions:

\begin{proposition}
\label{singabs}Consider system $(A_{w})$ with $\mathcal{A}_{p}=\Delta_{p}$ and
$\mathcal{A}_{q}=\Delta_{q}.\medskip$

\noindent(i) If $\mu<\frac{(q+b)(p-1)}{N-p},$ there exist solutions such that
\[
\lim_{\rho\rightarrow0}\rho^{\frac{N-p}{p-1}}u=\alpha>0,\qquad\lim
_{\rho\rightarrow0}v=\beta>0.
\]
$\medskip$

\noindent(ii) If $\delta<\frac{(N+a)(q-1)}{N-q}$ and $\mu<\frac{(N+b)(p-1)}%
{N-p},$ there exist solutions such that
\[
\lim_{\rho\rightarrow0}\rho^{\frac{N-p}{p-1}}u=\alpha>0,\qquad\lim
_{\rho\rightarrow0}\rho^{\frac{N-q}{q-1}}v=\beta>0.
\]
$\medskip$

\noindent(iii) If $\gamma_{a,b}>\frac{N-p}{p-1},$ and either $\mu
>\frac{(N+b)(p-1)}{N-p}$ or $\mu<\frac{(q+b)(p-1)}{N-p}$, there exist
solutions such that
\[
\lim_{\rho\rightarrow0}\rho^{\frac{N-p}{p-1}}u=\alpha>0,\qquad\lim
_{\rho\rightarrow0}\rho^{\frac{1}{q-1}(\frac{N-p}{p-1}\mu-(q+b))}%
v=\beta(\alpha)>0.
\]
$\medskip$

\noindent The results extend by symmetry, after exchanging $u,v,a,\gamma
_{a,b}$ and $v,u,b,\xi_{a,b}.$
\end{proposition}

\begin{proof}
As in \cite{BG}, \cite{BGY3} we prove the existence of trajectories of system
$(\Sigma)$ and return to $u,v$ by using (\ref{form}).$\medskip$

\noindent(i) Such solutions correspond to trajectories converging to the fixed
point $G_{0}=(\frac{N-p}{p-1},0,0,N+b-\frac{N-p}{p-1}\mu)$ of $(\Sigma).$ The
linearization at $G_{0}$ gives the eigenvalues
\[
\lambda_{1}=\frac{N-p}{p-1}>0,\;\lambda_{2}=\frac{1}{q-1}(q+b-\frac{N-p}%
{p-1}\mu),\;\lambda_{3}=N+a>0,\;\lambda_{4}=\frac{N-p}{p-1}\mu-N-b.
\]
If $\mu<\frac{(q+b)(p-1)}{N-p}$, then $\lambda_{2},\lambda_{4}<0$. Then
$\mathcal{V}_{u}$ has dimension 3, and $\mathcal{V}_{u}\cap\left\{
Y=0\right\}  $ and $\mathcal{V}_{u}\cap\left\{  Z=0\right\}  $ have dimension
2. This implies that ${\mathcal{V}}_{u}$ must contain trajectories such that
$Y,Z<0<X,W$.$\medskip$

\noindent(ii) Such solutions correspond to the fixed point $A_{0}=\left(
\frac{N-p}{p-1},\frac{N-q}{q-1},0,0\right)  .$ All the eigenvalues are
positive:
\[
\lambda_{1}=\frac{N-p}{p-1},\;\lambda_{2}=\frac{N-q}{q-1},\;\lambda
_{3}=N+a-\delta\frac{N-q}{q-1},\;\lambda_{4}=N+b-\mu\frac{N-p}{p-1}.
\]
The unstable manifold $\mathcal{V}_{u}$ has dimension $4,$ then there exists
an infinity of trajectories converging to $A_{0}$ with $X;Y,Z,W<0$.$\medskip$

\noindent(iii) Such solutions correspond to the fixed point $P_{0}=\left(
\frac{N-p}{p-1},Y_{\ast},0,W_{\ast}\right)  ,$ with
\[
Y_{\ast}=\frac{1}{q-1}(\frac{N-p}{p-1}\mu-(q+b)),\qquad W_{\ast}%
=N+b-\frac{N-p}{p-1}\mu.
\]
The eigenvalues are given by
\[
\lambda_{1}=\frac{N-p}{p-1}>0,\quad\lambda_{2}=Y_{\ast},\quad\lambda_{3}%
=\frac{D}{q-1}(\gamma-\frac{N-p}{p-1})>0,\quad\lambda_{4}=-W_{\ast}.
\]
If $\mu>\frac{(N+b)(p-1)}{N-p}$, then $\lambda_{2},\lambda_{4}>0$ and thus
$\mathcal{V}_{u}$ has dimension $4$, then there exist trajectories, with
$X,Y,Z,W<0,$ converging to $P_{0}$. If $\mu<\frac{(q+b)(p-1)}{N-p}$, then
$\lambda_{2},\lambda_{4}<0$, $\mathcal{V}_{u}$ has dimension $2$, and
$\mathcal{V}_{u}\cap\left\{  Z=0\right\}  $ has dimension $1$, thus there also
exist trajectories with $X,Z,W<0<Y$ converging to $P_{0}.\medskip$
\end{proof}

In the same way, system $(M_{w})$ has a particular singular solution when
$\gamma_{a,b}>\frac{N-p}{p-1}$ and $\xi_{a,b}<\frac{N-q}{q-1},$ and we find
other singular solutions:

\begin{proposition}
\label{singmix}Consider system $(M_{w})$ with $\mathcal{A}_{p}=\Delta
_{p},\mathcal{A}_{q}=\Delta_{q}.\medskip$

\noindent(i) If $\gamma_{a,b}>\frac{N-p}{p-1},$ and $\xi_{a,b}>\frac{N-q}%
{q-1},$ there exist solutions such that%
\[
\lim_{\rho\rightarrow0}\rho^{\frac{N-q}{q-1}}v=\beta>0,\qquad\lim
_{\rho\rightarrow0}\rho^{\frac{1}{p-1}(\frac{N-q}{q-1}\delta-(q+a))}%
u=\beta(\alpha)>0.
\]
$\medskip$

\noindent(ii) If $\delta<\frac{(N+a)(q-1)}{N-q}$ and $\mu<\frac{(N+b)(p-1)}%
{N-p},$ there exist solutions such that
\[
\lim_{\rho\rightarrow0}\rho^{\frac{N-p}{p-1}}u=\alpha>0,\qquad\lim
_{\rho\rightarrow0}\rho^{\frac{N-q}{q-1}}v=\beta>0.
\]

\end{proposition}

\begin{proof}
(i) These solutions correspond to the fixed point $Q_{0}$ deduced from $P_{0}$
by symmetry, and our assumptions imply $\delta>\frac{(N+a)(q-1)}{N-q},$ hence
there exist trajectories, such that $X,Y,Z<0<W$ converging to $Q_{0}.\medskip$

\noindent(ii) The conclusion follows as in Proposition \ref{singabs},
(ii).$\medskip$
\end{proof}

We refer to \cite{BG} and \cite{BG2} for a description of all the (various)
possible behaviours of the solutions in the case $p=q=2.\medskip$

\textbf{Acknowledgments} The authors thank the anonymous referees for their
relevant remarks and suggestions which have improved the final form of the manuscript.

The first author was supported by Fondecyt 1110268 and Ecos-Conicyt C08E04.
The second and the third authors were supported by Fondecyt 1110003 and
1110268, as well as Ecos-Conicyt C08E04.

\end{document}